\documentclass{amsart}
\usepackage{graphicx}
\usepackage{amsfonts}
\usepackage{amscd}
\usepackage{amssymb}
\usepackage{alltt}
\usepackage{multicol}
\usepackage{amsmath}
\usepackage{amsthm}
\usepackage{amscd}

\usepackage[numbers]{natbib}
\usepackage{rotating}
\usepackage{floatpag}
 \rotfloatpagestyle{empty}
\usepackage{amsthm}
\usepackage{graphicx}
\usepackage{multind}
\usepackage{times}

\usepackage{verbatim}
\usepackage{latexsym}
\usepackage{amsfonts}
\usepackage{amsmath}
\usepackage{crop}
\usepackage{fancyhdr}
\usepackage{txfonts}
\usepackage{setspace}
\usepackage{ellipsis} %
\usepackage{wrapfig}
\usepackage{listings}

\usepackage{tikz} 
\usetikzlibrary{chains,shapes,arrows,trees,matrix,positioning,backgrounds,fit,calc,fadings}

\usepackage[mathscr]{euscript} 
\usepackage{pifont} 
\usepackage[displaymath]{lineno}

\crop

\theoremstyle{plain}
\newtheorem{theorem}[equation]{Theorem}
\newtheorem{lemma}[equation]{Lemma}
\newtheorem{corollary}[equation]{Corollary}
\newtheorem{example}[equation]{Example}
\newtheorem{proposition}[equation]{Proposition}
\theoremstyle{definition}
\newtheorem{definition}[equation]{Definition}
\theoremstyle{remark}
\newtheorem{remark}[equation]{Remark}
\numberwithin{equation}{subsection}

\renewcommand\footnotemark{}

\begin{document}
\title
    {The Spherical Hecke algebra, partition functions, and motivic integration}
\author{William Casselman, Jorge E. Cely, and Thomas Hales}

\begin{abstract}   
This article gives a proof of the Langlands-Shelstad fundamental 
lemma for the spherical Hecke algebra
for every unramified $p$-adic reductive group $G$ in large positive 
characteristic.  The proof is based on the transfer principle
for constructible motivic integration.    To carry this out,
we introduce a general family of partition functions attached to 
the complex $L$-group of the unramified $p$-adic group $G$.
Our partition functions specialize to Kostant's $q$-partition 
function for complex connected groups and also 
specialize to the Langlands $L$-function of a spherical representation.
These partition functions are used to extend numerous 
results that were previously known only when the $L$-group is
connected (that is, when the $p$-adic group is split).
We give explicit formulas for branching rules, the inverse of the weight multiplicity matrix,
the Kato-Lusztig formula for the inverse Satake transform,
the Plancherel measure,  and Macdonald's formula for the spherical Hecke algebra on a
non-connected complex group (that is, non-split unramified $p$-adic group).  
%
\end{abstract}

 \lhead{Casselman, Cely, Hales}
\rhead{Hecke algebras, partition functions, and motivic integration}

\parskip=\baselineskip

 \maketitle



%


\newcommand{\ring}[1]{\mathbb{#1}}
\newcommand{\ang}[1]{\langle{#1}\rangle}
\def\op#1{{\operatorname{#1}}}

\def\Q{{\ring{Q}}}

\def\C{\mathcal C}
\def\N{\mathcal N}
\def\H{\mathcal H}
\def\T{\mathcal T}
\def\D{\mathcal D}

\def\n{{\mathfrak n}}
\def\g{{\mathfrak g}}
\def\t{{\mathfrak t}}
\def\h{{\mathfrak h}}

\def\inv{\op{inv}}
\def\dom{P^+}
\newcommand{\card}{\op{card}}
\def\Frob{\op{Frob}}
\def\dotw{\dot w}
\def\uu{\upsilon} 

\def\libel#1{{\text{\sc [#1]~}}\label{#1}}
\def\rif#1{(\ref{#1}-{\text{\sc #1})}}

Let $F$ be a $p$-adic field; that is, a finite extension of
$\ring{Q}_p$ or $\ring{F}_p((t))$.  Let $G$ be an unramified reductive
group and $H$ an unramified endoscopic group of $G$, both defined over
$F$.  Let $\H(G)$ and $\H(H)$ be the spherical Hecke algebras on $G$
and $H$.  Associated with a morphism $\xi:{}^LH\to {}^LG$ of
$L$-groups, there is a homomorphism $b_\xi:\H(G)\to \H(H)$, obtained
by composing three maps: the Satake transformation of $\H(G)$, the
pullback under $\xi$, and the inverse Satake transformation to
$\H(H)$.

Let $A$ be a maximal split torus of $G$.  Let 
\[
\dom = \{\lambda\in X_*(A) \mid \ang{\lambda,\alpha^\vee}\ge 0,
\quad \alpha\in \Delta_1\}
\]
be a positive chamber of $X^*(\hat S) = X_*(A)$.  (The notation is
explained in Section \ref{sec:L}.)  The spherical Hecke algebra
$\H(G)=\H(G//K)$ of complex-valued functions that are biinvariant with
respect to a given hyperspecial subgroup $K$ has a linear basis given
by characteristic functions $f_\lambda$ of double cosets
$K\varpi^\lambda K$.  Here $\varpi$ is a fixed uniformizing element
and $\lambda$ runs over the cocharacters in $P^+$.

In this article, we use motivic integration to study the spherical
Hecke algebra and the function $B_\xi:P^+\times H(F)\to \ring{C}$,
given by $(\lambda,h)\mapsto b_\xi(f_\lambda)(h)$.  This function is
the specialization of a constructible motivic function (Theorem
\ref{thm:B}).

As an application, we show that the fundamental lemma for the
spherical Hecke algebra falls within the scope of the transfer
principle for constructible motivic functions
(Section~\ref{sec:transfer}).  This implies that the fundamental lemma
holds for the spherical Hecke algebra over fields of large positive
characteristic (Theorem \ref{thm:fl}).

This application to the fundamental lemma is the main motivation for
this work.  Our results overlap with those of Bouthier, who proves the
fundamental lemma for the spherical Hecke algebra in positive
characteristic under the restrictions that the group $G$ is semisimple
and simply connected, and the endoscopic group is split
\cite[Theorem~0.2]{bouthier}.  Our proof of the fundamental lemma for
the spherical Hecke algebra holds without restriction on the group and
endoscopic group.  Unlike Bouthier, we are unable to be explicit in
our assumption on the characteristic of the field.  This is an
unfortunate limitation of the methods we use.  In other work, Lemaire,
Moeglin, and Waldspurger propose that the method of close fields might
be used to transfer the fundamental lemma for the spherical Hecke
algebra from characteristic zero to positive characteristic, but as
far as we know, this has not been carried out~\cite[\S1.3]{LMW}.

The construction of $B_\xi$ passes through the Langlands dual ${}^LG$,
which is a non-connected complex reductive group.  Our
constructibility result for $B_\xi$ follows from the Presburger
constructibility of various functions on lattices in the dual:
Macdonald's formula, weight multiplicity formulas, the inverse of the
weight multiplicity matrix, the Plancherel measure, the geometric
Satake transform, and the Kato-Lusztig formula for the inverse Satake
transform.  When $G$ and $H$ are split, we can take ${}^LG = \hat G$
and ${}^LH=\hat H$ to be connected.  In this case, the desired
formulas were previously known.  In this article, we generalize these
formulas to non-connected complex reductive groups.  These
generalizations are a major part of this work.

One novelty of this work is that we show how to extend the theory of
motivic integration to the Langlands dual group, by encoding
representation-theoretic data of the complex dual group as Presburger
constructible functions on the character lattice. These Presburger
functions can then be recombined with constructible functions on the
$p$-adic group.  A second innovation is to encode the entire Hecke
algebra into a single constructible function $B_\xi$.  This makes it
possible to invoke the the transfer principle of motivic integration a
single time, rather than once for each function in the Hecke algebra.
(Invoking the transfer principle an infinite number of times could
potentially leave us with nothing, because we lose finitely many
primes with each invocation.)

A framework for studying the spherical Hecke algebra through motivic
integration is provided by Cely's thesis \cite{cely}.  This article
builds on that work.  The second author would like to
express gratitude to his thesis advisor T. Hales for all the
teaching and support.  We thank Julia Gordon, who served on Cely's
thesis committee and who provided valuable suggestions.

\section{the Satake transform, 
Macdonald's formula, and related topics}

In this section, we extend various results from split $p$-adic groups
to unramified groups (from complex connected reductive groups to
non-connected groups on the $L$-group side).

\subsection{root systems}

Let $G$ be an unramified reductive group over a $p$-adic field $F$.
It is defined by descent from a finite unramified extension $E/F$ over
which $G$ splits.  It is determined by an automorphism $\theta$ of the
root datum $(X^*,\Phi,X_*,\Phi^\vee)$ of the split form $G^*$ of $G$.
The automorphism $\theta$ has finite order and preserves a set of
simple roots associated to some set $\Phi^+$ of positive roots.  If
$\op{Frob}$ is the Frobenius automorphism of $E/F$, the group $G/F$ is
defined by twisting the action of Frobenius on $G(E)$ by $\theta$.  Let
$(B,T,X)$ be a pinning of $G$ over $F$, preserved by $\theta$.

Let $A$ be a maximal split torus in $T$.  For any $R\subseteq \Phi$,
let $M_R$ be the centralizer of
\[
A_R = \{a\in A\mid \alpha(a)=1,\quad \alpha\in R\}.
\]
We have $X_*(A) = X_*(T)^\theta$ and $X^*(A) =
X^*(T)/(1-\theta)X^*(T)$.  The pairing of $X^*(A)$ and $X_*(A)$ is
induced naturally from that of $X^*(T)$ and $X_*(T)$.

The image of $\Phi$ in $X^*(A)$ is the restricted root system $\Phi_{res}$.
It is well known that $\Phi_{res}$ is indeed a root system, and this is
easy to verify directly in this context.  The roots of $\Phi_{res}$ are in
bijection with the $\theta$-orbits of roots in $\Phi$, and this
bijection restricts to a bijection between simple roots in $\Phi_{res}$
and orbits of simple roots in $\Phi$.  If $\alpha$ is a root of $\Phi$
(or coroot), we write $[\alpha]=\{\alpha,\theta\alpha,\ldots\}$ for
its $\theta$-orbit, often identified with a root of $\Phi_{res}$.
A root $\alpha\in\Phi_{res}$ is {\it indivisible} if $\frac12\alpha$ is not a root.
Let $\Phi^{red}$ and 
Let $\Phi_{red}$ be the set of nonmultipliable and
indivisible roots of $\Phi_{res}$.  
They are both reduced root systems.
Two roots $\alpha$ and $\alpha'\in\Phi_{res}$ are {\it homothetic} if
$\alpha' = k\alpha$ for some $k>0$.  Every root is homothetic to an
indivisible root.  

Each root $\beta\in\Phi$ can be assigned a diagram $A_1$ or $A_2$ as
follows.  The construction is best described in the split group $G^*$.
Let $S=[\beta']$ be the $\theta$-orbit that corresponds to the
indivisible root homothetic to $[\beta]$.  The simple positive roots
of $M_S$ form a single $\theta$-orbit $S$.  We consider its Dynkin
diagram.  By the transitivity of $\theta$, all components of the
diagram have the same type, either $A_1$ or $A_2$.  This is the
diagram of $\beta$.  This construction can be applied to $(G^*,\Phi)$
or to $(\hat G,\Phi^\vee)$, where $\hat G$ is the complex dual.  A
coroot $\alpha^\vee$ has the same diagram as the root $\alpha$.  We let
$b(\beta)$ be the number of connected components of the Dynkin
diagram of $M_S$.  (According to the types of Kottwitz and Shelstad, type I
means diagram $A_1$, type II means a simple root in diagram $A_2$, and
type III means a highest root in diagram $A_2$
\cite{kottwitz1999foundations}.)

Let $N:X_*(T)\to X_*(A)$ be the norm map: 
$N\alpha = \sum_{\alpha'\in  [\alpha]} \alpha'$.

\begin{lemma}\label{lemma:norm}
 If $[\alpha]\in\Phi_{red}$ is an indivisible restricted root, then
  the corresponding coroot is
\begin{equation}\label{eqn:norm}
[\alpha]^\vee = k\, N\alpha^\vee
\end{equation}
when the diagram of $\alpha$ has type $A_k$, for $k=1,2$.
\end{lemma}

\begin{proof}
(See \cite[1.3.9]{kottwitz1999foundations}.)
\end{proof}

Let $W$ be the Weyl group attached to the root datum of $G$ over a
splitting field $E$.  The restricted Weyl group $W^\theta$ is the
subgroup of $W$ commuting with $\theta$.  The group $W^\theta$ is a
Coxeter group.  The simple reflection in $W^\theta$ associated with an
orbit $[\alpha]$ of simple roots in $\Phi$ is the longest element in
the Weyl group of the Levi component $M_{[\alpha]}$.  We write
$\ell(w)$ for the length of $w\in W^\theta$, computed relative to the
set of simple reflections of the Coxeter group $W^\theta$.

\subsection{$L$-groups}\label{sec:L}

We review some aspects of the theory of non-connected complex
reductive groups from Steinberg \cite{steinberg1968endomorphisms},
Springer \cite{springer2010linear}, Kottwitz and Shelstad
\cite{kottwitz1999foundations}, Haines \cite{haines2016dualities}, and
Chriss \cite{chriss}.

Let ${}^LG = \hat G \rtimes \ang{\theta}$ be the $L$-group of the
unramified group $G$.  It has root system dual to that of $G$.  It is
a semidirect product of a connected complex reductive group $\hat G$
and a finite cyclic group generated by an outer automorphism $\theta$
of $\hat G$.  The automorphism $\theta$ preserves a pinning $(\hat
B,\hat T,\hat X)$ of $\hat G$. Let $\hat B = \hat T\hat N$.

Let $\Psi=\Phi^\vee$ and $\Psi^\vee=\Phi$ be the root and coroot
systems of the complex group $\hat G$ with respect to $\hat T$, and
let $\Psi^+$ be the set of positive roots with respect to $(\hat
T,\hat B)$.  If $R$ is any root system with positive roots $R^+$, we
let $\rho(R^+) = (1/2)\sum_{R^+} \alpha$ be the half-sum of positive
roots in $R^+$.  We have $\rho(\Psi^{red,+}) = \rho(\Psi^+)$.

The torus $\hat T$ is $\theta$-stable.  We form the quotient $\hat S =
\hat T/(1-\theta) \hat T$, where
\[
(1-\theta)\hat T = \{ t\theta(t^{-1}) \mid t\in \hat T\},
\]
using Steinberg's additive notation for a multiplicative group.  If
$\lambda\in X^*(\hat T)^\theta$ is $\theta$-fixed, then it is trivial
on $(1-\theta)\hat T$ and descends to a character $\lambda\in X^*(\hat
S)$.  This gives $X^*(\hat T)^\theta = X^*(\hat S)$.

We abbreviate $Y^* = X^*(\hat S)$. Let $O=O_F$, the ring of integers
of $F$.  We have identifications
\begin{equation}\label{eqn:identify}
T/T(O)=A/A(O)=X_*(A)=X_*(T)^\theta  =X^*(\hat T)^\theta = Y^* = X^*(\hat S).
\end{equation}
Each unramified character $\chi:T\to \ring{C}^\times$, by these
identifications, is a homomorphism
\begin{equation}
\chi\in\op{Hom}(T/T(O),\ring{C}^\times) = 
\op{Hom}(X^*(\hat S),\ring{C}^\times) = 
X_*(\hat S)\otimes \ring{C}^\times = \hat S.
\end{equation}
We write $\chi = \chi_s$, for $s\in\hat S$.

Let $e^\lambda$ be the basis element of the group algebra
$\ring{C}[Y^*]$ of $Y^*$, indexed by $\lambda\in Y^*$.

Recall that
\[
\dom = \{\lambda\in X_*(A) \mid \ang{\lambda,\alpha^\vee}\ge 0,
\quad \alpha\in \Delta_1\}
\]
where $\Delta_1=\Delta(\Psi^{red})$ is a set of simple roots for
$\Psi^{red}$.  A basis of $W^\theta$-fixed functions in
$\ring{C}[Y^*]$ is
\begin{equation}\label{eqn:mmu}
m_\mu = \sum_{\lambda\in W^\theta(\mu)} e^\lambda, 
\quad \text{for }\mu\in P^+,
\end{equation}
where $W^\theta(\mu)$ is the orbit of $\mu\in\ring{C}[Y^*]$ under
$W^\theta$.

Let $\hat G_\theta$ be the complex group with Cartan subgroup $\hat
S$, root system $\Psi^{red}=(\Phi_{red})^\vee$, and root datum
\[
(X_*(A),(\Phi_{red})^\vee,X^*(A),\Phi_{red}) 
= (X^*(\hat S),\Psi^{red},X_*(\hat S),\Psi^{red,\vee}).
\]
The complex group $\hat G_\theta$ is the $L$-group of $G^1$, the
identity component of the $\theta$-fixed subgroup of the split form
$G^*$ \cite[\S1.3]{kottwitz1999foundations}.  As we will see, $\hat
G_\theta$ appears naturally in the twisted Weyl character formula for
${}^LG$, hence also in the Satake transform and its inverse.

\begin{example} Let $G=\op{SU}(n,E/F)$ be an unramified unitary group
  in an odd number of variables $n=2k+1$.  The automorphism $\theta$
  has order $2$.  The cocharacter groups of $T$ and $A$ are
\[
X_*(T) = \{(t_1,\ldots,t_{n})\in \ring{Z}^n\mid t_1+\cdots t_n=0\}, 
\quad
X_*(A)  = \ring{Z}^k,
\]
with identification $(t_1,\ldots,t_k)\in X_*(A)\mapsto
(t_1,\ldots,t_k,0,-t_k,\ldots,-t_1)\in X_*(T)^\theta$.  Following
Lemma \ref{lemma:norm}, we compute norms to get
\[
\Psi^{red} = \{\pm t_i\pm t_j\mid i\le j\le k\}\subset X_*(A).
\]
We recognize the root datum as that of $\hat G_\theta =
\op{Sp}(2k,\ring{C})$ with root system $\Psi^{red}$.  In this case, $G^1
= \op{GL}(n)^{\theta,0} = \op{SO}(2k+1)$, which indeed has $L$-group
$\hat G_\theta$.
\end{example}

\subsection{the partition function}

Let $\theta_1\in N_{\hat G}(\hat T) \rtimes \ang{\theta}$ be an
element of finite order.  Set $\hat T_1 = \hat T/(1-\theta_1)\hat T$
and $X^*(\hat T_1) = X^*(\hat T)^{\theta_1}$.  Let $N_1:X^*(\hat T)\to
X^*(\hat T)^{\theta_1}$ be the norm map with respect to $\theta_1$:
\[
N_1 \mu = \sum_{\mu'\in \ang{\theta_1}\mu} \mu'.
\]

Let $V$ be a finite dimensional representation of ${}^LG$, with
weight space decomposition 
$V=\oplus V_\mu$.  We have
$\theta_1(V_\mu) = V_{\theta_1\mu}$.  Let $R\subset X^*(\hat T)$ be a
$\theta_1$-stable set of weights of $V$, and set
\begin{equation}\label{eqn:VR}
V_R = \oplus_{\mu\in R} V_\mu.
\end{equation}

We define a symbolic operator $E$ on $V_R$ that is diagonal with
respect to the weight space decomposition:
\[
E v = e^{\mu} v, \text{ for } v \in V_\mu.
\]
(We warn the reader that $E$ denotes an unramified field extension $E/F$
or this symbolic operator, depending on the context.)
We define a $q$-determinant $D(\hat G,V_R,\theta_1,E,q)$ and a
$q$-partition function $P(\hat G,V_R,\theta_1,E,q)$ as
\begin{align}\label{eqn:det}
D(\hat G,V_R,\theta_1,E,q) &= \det(1- q \theta_1 E;V_R);\\ 
P(\hat G,V_R,\theta_1,E,q) &= D(\hat G,V_R,\theta_1,E,q)^{-1}.
\end{align}
When $V$ is the adjoint representation, we sometimes abbreviate
$D(\hat G,V_R,\theta_1,E,q)$ to $D(\hat G,R,\theta_1,E,q)$.  The
determinant and partition function carry the same information, and we
pass back and forth between $D$ and $P$ according to convenience.

We may view the determinant and partition functions as functions on
$\hat T$, by evaluating each $e^\mu (t) = \mu(t)$, for $t\in \hat T$,
so that
\[
\det(1- q\theta_1 E;V_R)(t) = \det(1-q\theta_1 t;V_R).
\]
Taking $\theta_1^{-1}(u) u^{-1}\in (1-\theta_1)\hat T$, we have
\begin{align*}
\det (1-q \theta_1 t \theta_1^{-1}(u) u^{-1};V_R) 
&= \det(1- u (q \theta_1 t ) u^{-1};V_R) = \det(1-q\theta_1 t;V_R).
\end{align*}
Thus, the partition function and determinant descend to functions on
$\hat T_1$.  

For $w\in W^\theta$ and $v\in V_\beta$, define $E^w$ by $E^w v =
e^{w\beta} v$.  If $\dot w$ is a representative of $w$ in the
normalizer of $\hat T$, then
\[
(\dot w^{-1} E \dot w) v = E^w v.
\]
Also, for $t\in \hat T$,
\[
(E v) (w^{-1} t w) = (E^w v) (t) = (w\beta)(t) v.
\]
Write $E^{-w} = (E^{-1})^w = (E^w)^{-1}$.

The next lemma gives the general shape of a factorization
of the determinant.

\begin{lemma}\label{lemma:fact}  
  Let $R = \{\theta^i\mu\}$ be the orbit of a single weight of $V$.
  Then $D(\hat G,V_R,\theta_1,E,q)$ is a finite product of factors of
  the form
\[
1 - \zeta q^b e^{N_1\mu},
\]
where $b$ is the cardinality of the orbit $R$ and $\zeta^k=1$, where
$k b$ is the order of $\theta_1$.
\end{lemma}

\begin{proof} Fix $\mu_0\in R$ and let $\mu_i = \theta_1^i \mu_0$.
  The abelian group $\ang{\theta_1^b}$ acts on $V_{\mu_0}$.  Write
  $V_{\mu_0} = \oplus W$, where each $W$ is a $1$-dimensional
  representation of $\ang{\theta_1^b}$.  Then $V_R = \oplus
  \ang{\theta_1} W$, where $\ang{\theta_1} W = \oplus_{i=0}^{b-1}
  \theta_1^i W$.  We have $\theta_1^b v_0 = \zeta v_0$, for $v_0\in W$
  and some $\zeta = \zeta_W$.  Let $v_i = \theta_1^i v_0$.  The
  operator $1 - q \theta_1 E$ on the summand $\ang{\theta_1} W$ with
  respect to this basis is
\[
\begin{pmatrix}
1 & -q e^{\mu_0} & 0 & \ldots\\
0 & 1 & -q e^{\mu_1} & \ldots\\
   & \ldots & & \\
-\zeta  q e^{\mu_{b-1}} & 0 & \ldots & 1
\end{pmatrix}.
\]
The result follows by taking its determinant.
\end{proof}

\subsubsection{relation with $L$-functions}

Langlands has defined an $L$-function for spherical representations of
$G$.  It is written $L(\pi,V,q^{-s})$, where $\pi$ is an irreducible
admissible representation of $G$ with a $K$-fixed vector, and $V$ is a
representation of the $L$-group ${}^LG$, which we assume factors
through some finite unramified Galois extension $\op{Gal}(E/F)$.

Let $t_\pi\in \hat S$ be the Frobenius-Hecke parameter of $\pi$.  We
let $\theta_1 = \theta = \Frob$, the automorphism of $\hat G$ coming
from the action of Frobenius on the root datum.  We let $R$ be the set
of all weights, so that $V = V_R$.  As observed above, the partition
function $P(\hat G,V,\theta,E,q)$ is a function on $\hat T_1 = \hat
S$.  

\begin{lemma} 
  The partition function evaluates to the local $L$-function.  More
  precisely,
\[
P(\hat G,V,\theta,E,q_F^{-s})(t_\pi) = L(\pi,V,q_F^{-s}).
\]
\end{lemma}

\begin{proof} Both sides are defined as the reciprocal of a
  determinant.  On both sides it is the determinant of the same
  element acting on the same vector space.
\end{proof}

\subsubsection{the partition function 
for the adjoint representation}\label{sec:adjoint}

A case of particular importance for us is the following.  Let $\g$ be
the adjoint representation of $\hat G$.  It is an irreducible highest
weight representation whose highest weight is $\theta$-fixed.  Hence
$\g$ extends to an irreducible representation of $\hat
G\rtimes\ang{\theta}$.  Let $\n$ be the Lie algebra of $\hat N$.  Then
$\n = \g_R$, where $R$ is the set of positive roots.  The set $R$ is
$\theta$-stable.  We have a partition function
\[
P(\hat G,\n,\theta,E,q).
\]
Much of this article handles this particular case.  When $\hat G$,
$\n$, and $\theta$ are fixed, we abbreviate $P(E,q) = P(\hat
G,\n,\theta,E,q)$.

Recall that each $\alpha\in \Psi^{red,+}$ has the form
$[\beta^\vee]^\vee$, with $[\beta^\vee]\in \Phi_{red}$ and some
$\beta$ in $\Psi^+$.  The root $\beta$ has a diagram $A_1$ or $A_2$,
associated constant $b=b(\beta)$, and $k N\beta = \alpha$ for diagram
type $A_k$ (Lemma \ref{lemma:norm}).  For each $\alpha\in \Psi^{red}$,
we define
\begin{equation}\label{eqn:d}
d_\alpha(q) =
\begin{cases} {1-q^b e^\alpha},    &\text{if diagram } A_1;\\
{(1-q^{2b} e^{\alpha/2})(1+q^b e^{\alpha/2})},
&\text{if diagram } A_2.
\end{cases}
\end{equation}
We have the following factorization refining Lemma \ref{lemma:fact}.

\begin{lemma} \label{lemma:prod}
The determinant factors as
\[
\det(1-\theta  E q ;\n) = \prod_{\alpha\in\Psi^{red,+}} d_\alpha(q).
\]
\end{lemma}

\begin{proof} Similar factorizations in the special case $q=1$ are
  found in \cite{jantzen1977darstellungen}, \cite{wendt2001weyl}.
  Here is a sketch.  The determinant is block diagonal, with a block
  for each $\theta$-orbit in $\Psi^+$.

  We first consider diagram type $A_1$.  Write $\alpha = N\beta$, as
  above with $\beta\in\Psi$.  On the block $[\beta]$, we can pick a
  basis $X_i$ of the root spaces $\n_{\theta^i\beta}$ such that
  $\theta$ acts as $\theta X_i = X_{i+1}$, with indices mod $b$.  The
  determinant restricted to this block satisfies the fool's identity
  $\det(I-A) = \det(I)-\det(A)$, which yields $1- \prod_{\beta'\in
    [\beta]} {q e^{\beta'}} = 1- q^b e^{N \beta}$.

  Next consider diagram type $A_2$.  Write $\alpha = 2N\beta$ as
  above.  We choose three positive roots $\beta,\beta',\gamma\in\Psi$
  forming the positive root system of $A_2$, where
  $\gamma=\beta+\beta'$ is the highest root.  We have
  $N\gamma=N\beta=N\beta'=\alpha/2$.  Recall that $b$ is the number of
  connected components in the Dynkin diagram.  Then $\theta^b$
  preserves the $A_2$ factor and $\theta^b(\beta)=\beta'$.  There are
  two $\theta$-orbits of roots: $[\gamma]$ and $[\beta]=[\beta']$ of
  cardinalities $b$ and $2b$.  We may pick root vectors
  $X_{\beta},X_{\beta'},X_{\gamma}$ in the root spaces of
  $\beta,\beta',\gamma$ such that
\[
\theta^b(X_\beta)= X_{\beta'},\quad \theta^b(X_{\beta'})=X_\beta,\quad
\theta^b X_\gamma = \theta^b [X_\beta,X_{\beta'}] 
= [X_{\beta'},X_\beta] = -[X_\beta,X_{\beta'}] = -X_\gamma.
\]
Note the sign $(-1)$ that appears for the orbit $[\gamma]$.  We choose
root vectors on the entire orbits by $\theta^i X_\delta=X_{\theta^i
  \delta}$, for $i=1,\ldots,b-1$ and
$\delta\in\{\beta,\beta',\gamma\}$.  We compute the determinant on
these two orbits as before.  For the orbit $[\gamma]$, we obtain
$1+\prod_{\gamma'\in [\gamma]} {q e^{\gamma'}} = 1+ q^b e^{\alpha/2}$.
The orbit $[\beta]$ gives $1-q^{2b} e^{\alpha/2}$.
\end{proof}

In the special case $\theta=1$, the matrix $q \theta E$ is diagonal,
every orbit has cardinality $1$, and the partition function is a
product over positive roots $P(E,q) = \prod_{\beta\in\Psi^+} (1- q
e^\beta)^{-1}$.  This is the classical $q$-partition function.

\begin{corollary}[twisted Weyl denominator]\label{cor:prod1} 
Specializing to $q=1$, we have
\[
P(E,1) = \prod_{\alpha\in\Psi^{red,+}} (1-e^{\alpha})^{-1}.
\]
\end{corollary}

\begin{proof}  
Set $q=1$ in the lemma and observe that $d_\alpha(1)=1-e^\alpha$.
\end{proof}

\begin{corollary}\label{cor:weyl-p}  
for all $w\in W^\theta$,
\[
P(E^w,q) P(E^{-w},q) = P(E,q)P(E^{-1},q).
\]
\end{corollary}

\begin{proof} 
By the lemma,
\[
P(E,q)P(E^{-1},q) = \prod_{\alpha\in \Psi^{red}} d_{\alpha}(q)^{-1},
\]
as $\alpha$ runs over the full root system $\Psi^{red}$.  The result
follows by observing that $w\in W^\theta$ permutes $\Psi^{red}$,
preserving the diagram type $A_k$ and constant $b$ attached to each
root.
\end{proof}

\subsection{twisted Weyl character formula}\label{sec:weyl-char}

We review the proof of the Weyl-character formula, as presented in
\cite{kostant1961lie}, \cite{jantzen1977darstellungen},
\cite{wendt2001weyl}, and \cite{kumar2009characters}.  At the same
time, we consider various $q$-deformations of the standard formulas.

Let $\lambda$ be a dominant weight in $X^*(\hat T)^\theta = X^*(\hat
S)$.  Let $V_\lambda$ be the irreducible module of $\hat G$ with
highest weight $\lambda$.  The $\theta$-invariance of $\lambda$
implies that $V_\lambda$ extends uniquely to a representation of $\hat
G \rtimes \ang{\theta}$ such that $\theta v = v$ for $v$ in the
highest weight space of $V_\lambda$.  We let $\tau_\lambda$ be the
character on $V_\lambda$, restricted to $\hat G\rtimes\theta$.   The
$\hat G$-conjugacy class of $t\theta\in \hat T\rtimes\theta$ depends
only on the image of $t$ in $\hat S$.  Thus, we may consider
$\tau_\lambda\in \ring{C}[X^*(\hat S)] =\ring{C}[Y^*]$.  Furthermore,
$\tau_\lambda$ is $W^\theta$-invariant.

Let $\rho = \rho(\Psi^+) = \rho(\Psi^{red,+})$.  We define a dot
operator $w\bullet \mu = w(\mu+\rho)-\rho$, for $w\in W^\theta$ and
$\mu\in Y^*$.  We define an alt-symmetrizer operator
\[
J:\ring{C}[Y^*]\to \ring{C}[Y^*],
\quad J(f) = \sum_{w\in W^\theta} (-1)^{\ell(w)} w(f e^\rho) e^{-\rho}.
\]

\begin{theorem}[twisted Weyl character]  
  For every dominant weight $\lambda\in X^*(\hat T)^\theta = X^*(\hat
  S)$,  the irreducible representation $V_\lambda$ of ${}^LG$ has
  character $\tau_\lambda\in \ring{C}[Y^*]$ on $\hat G\rtimes \theta$,
  where
\[
\tau_\lambda = J(e^\lambda) P(E^{-1},1).
\]
\end{theorem}

The Weyl denominator $P(E^{-1},1)$ is computed in Corollary
\ref{cor:prod1}.  If we take $\lambda=0$, then $\tau_\lambda=1$, and
the Weyl character formula gives a second formula for the Weyl
denominator:
\begin{equation}\label{eqn:wd2}
1= J(1) P(E^{-1},1).
\end{equation}
It is a remarkable consequence of the Weyl character formula that the
twisted character $\tau_\lambda$ is identical to the irreducible
character of $\hat G_\theta$ with highest weight $\lambda$.

\begin{proof} 
  We follow Kostant \cite{kostant1961lie}.  Let $\n$ be the Lie
  algebra of $\hat N$, considered as a module of ${}^L\hat T=\hat
  T\rtimes\ang{\theta}$ by the adjoint representation, and let $\n'$
  be its contragredient.  We write $\chi_j$ for the character of the
  exterior power $\Lambda^j \n'$.  We write $\tilde\chi_q = \sum_j
  (-q)^j\chi_j$ for the $q$-graded virtual character on the sum of
  $\Lambda^j \n'$, with grading $(-q)^j$ on the $j$th summand.

  The character $\tilde\chi_q$ evaluated at $\theta t$ depends only on
  the image $s\in\hat S$ of $t\in \hat T$.  We have
\[
P(E^{-1},q)^{-1}(s) = \det(1- q\theta t;\n') 
= \sum_j (-q)^j \chi_j(t\theta) = \tilde\chi_q(t\theta).
\]
The sum is obtained from the determinant by picking a basis of
eigenvectors of $\theta t$ on $\n'$ and expanding into a polynomial in
$q$.  We have $E^{-1}$ rather than $E$ because $\n'$ is the
contragredient of $\n$.  This gives
\begin{equation}\label{eqn:tilde}
\tilde\chi_q P(E^{-1},q) = 1.
\end{equation}

Upon specialization to $q=1$, the spaces $C^j=\Lambda^j \n'\otimes
V_\lambda$ are the terms of a cochain complex of ${}^LT$-modules.  We
consider the virtual character $\tilde \tau_\lambda$ on the sum of
$C^j$, with grading $(-1)^j$ on $C^j$.  By an Euler-Poincar\'e
argument, $\tilde\tau_\lambda$ equals the character on the cohomology
of the complex.  The cohomology has been computed explicitly
\cite{kostant1961lie}.  These computations show that for each $w\in
W^\theta$, the weight $e^{w\bullet \lambda}$ occurs once in cohomology
with sign $(-1)^{\ell w}$.  Thus, in terms of the operator $J$, we
have
\[
\tilde \tau_\lambda = J(e^\lambda) 
= \sum_{w\in W^\theta} (-1)^{\ell(w)} e^{w\bullet\lambda}.
\]

By the description of $C^j$ as a tensor product, we have a product
decomposition $\tilde \tau_\lambda = \tilde \chi_{1}\tau_\lambda $
Multiplying both sides by $P(E^{-1},1)$, and using Equation
\ref{eqn:tilde}, we obtain the twisted Weyl character formula
\begin{equation}
\tau_\lambda = J(e^\lambda) P(E^{-1},1).
\end{equation}
\end{proof}

It is natural to extend $\tau_\lambda$ to a $q$-character by defining
$\tau_{\lambda,q}$ by $\tilde \tau_\lambda = \tilde
\chi_q\tau_{\lambda,q} $, so that
\begin{equation}\label{eqn:tauq}
\tau_{\lambda,q} = J(e^\lambda) P(E^{-1},q).
\end{equation}
We leave it as a research problem to find interpretations of
$\tau_{\lambda,q}$, along the lines of Kazhdan-Lusztig polynomials.
Kato and Lusztig give an answer when $\theta=1$.

\subsection{Macdonald's formula}\label{sec:macdonald}

Let $K$ be a hyperspecial maximal compact subgroup of $G$.  The
spherical Hecke algebra $\H(G//K)$ is the convolution
algebra of all complex-valued compactly support $K$-biinvariant
functions on $G$.  It has a linear basis consisting of characteristic
functions $f_\mu$ of double cosets $K\varpi^\mu K$, for $\mu\in P^+$.
By the Satake isomorphism, $\H(G//K)$ is isomorphic to the algebra
$\ring{C}[Y^*]^{W^\theta}$.  We write $\hat f_\mu\in \ring{C}[Y^*]$
for the image of $f_\mu$ under the Satake isomorphism.

We continue in the context of a complex group $\hat G \rtimes
\ang{\theta}$ and keep earlier notation.  As usual, we identify
elements of $\ring{C}[Y^*]$ with functions on $\hat S$.  As before
$\Phi=\Psi^\vee$ and $\Psi = \Phi^\vee$ are dual root systems.  Let
$\rho^\vee = \rho(\Phi^+)\in X_*(\hat T)$.

For each subset $S$ of the set $\Delta_1=\Delta(\Psi^{red})$ of simple
roots in $\Psi^{red,+}$, let $W_S\le W^\theta$ be the group generated
by the reflections in $S$.  Let $\ell' w$ be the length of $w\in W$
(as a function of the Weyl group $W$, and not the Weyl group
$W^\theta$).  Let $Q_S(q^{-1}) = \sum_{w\in W_S} q^{-\ell'w}$.  We
abbreviate $Q(q^{-1}) = Q_{\Delta_1}(q^{-1})$.  For each $\mu\in
P^+\subset Y^*$, let $S(\mu)$ be the subset of $\Delta_1$ such that
$\ang{\mu,\alpha^\vee}=0$ iff $\alpha\in S(\mu)$.

\begin{theorem}[Macdonald's formula]\label{thm:macdonald}
  Let $G$ be an unramified $p$-adic reductive group with $L$-group
  ${}^LG = \hat G \rtimes \ang{\theta}$.  For each $\mu\in P^+$, let
  $\hat f_\mu$ be the Satake transform of the characteristic function
  of $K\varpi^\mu K$, viewed as an element of
  $\ring{C}[Y^*]^{W^\theta}$.  Then
\[
\hat f_\mu = \frac{q^{\ang{\mu,\rho^\vee}}}{ Q_{S(\mu)}(q^{-1})} 
\sum_{w\in W^\theta} e^{w\mu} \frac{P(E^{-w},1)}{P(E^{-w},q^{-1})}.
\]
\end{theorem}

The partition function encodes the constants $q_\alpha$,
$q_{\alpha/2}$ that occur in the traditional formula
\cite{macdonaldspherical}.  Our formula is closely tied to the root
system $\Psi^{red}$ of $\hat G_\theta$ that occurs in the twisted Weyl
character formula for ${}^LG$.  A variant of Macdonald's formula of
this form was previously known when $G$ is split ($\theta=1$).

There is a related formula for the spherical function
$\Gamma:P^+\times\hat S\to\ring{C}$ that we mention.  For every
$\mu\in P^+$,
\begin{equation} 
\Gamma_\mu = 
\frac{q^{-\ang{\mu,\rho^\vee}}}{ Q(q^{-1})} 
\sum_{w\in W^\theta} e^{w\mu} \frac{P(E^{-w},1)}{P(E^{-w},q^{-1})}.
\end{equation}

\begin{proof}[Proof of Macdonald's formula]
  Our proof will relate the formula in the theorem to the standard
  form of Macdonald's formula.  Macdonald's formula is elaborated in
  \cite{casselman1980unramified} and \cite{casselman2005companion}.
  In the following discussion, we index terms by
  $\alpha\in\Phi_{red}$ (or $\alpha^\vee\in \Psi^{red} =
  (\Phi_{red})^\vee$), although terms in Macdonald's formula are
  traditionally indexed by $\Phi^{red}$, the set of nonmultipliable
  roots in $\Phi_{res}$.  This is a harmless change, because
  $\Phi_{red}$ is in natural bijection with $\Phi^{red}$ by
  sending each indivisible root $\alpha$ to the homothetic root
  $k\alpha\in\Phi^{red}$.

  Recall that for each $\alpha\in \Phi^+_{red}$, associated with an
  orbit $\alpha=[\beta]$ in $\Phi$, there is a diagram type $A_k$, for
  some $k\in\{1,2\}$, and cardinality $b$ (the number of connected
  components of the Dynkin diagram of $M_{[\beta]}$).  Casselman and
  Macdonald construct an element $a_{\alpha}\in T$, for each
  $\alpha\in\Phi_{red}$.  Let $\alpha_1 = \alpha^\vee\in \Psi^{red}$
  be its coroot.  By the explicit formulas in
  \cite{casselman2005companion}, for diagram type $A_k$, we have
  $a_\alpha = (k\alpha)^\vee(\varpi) = (\alpha_1/k)(\varpi)$.  Let
  $s\in \hat S$, and let $\chi_s\in\op{Hom}(T,\ring{C}^\times)$ be the
  associated unramified character.  Then
\[
\chi_s(a_{\alpha}) = \chi_s((\alpha_1/k)(\varpi)) = e^{\alpha_1/k}(s).
\]
Macdonald's formula is traditionally expressed in terms of the
constants $\chi_s(a_\alpha)$, which we rewrite in terms of
$e^{\alpha_1/k}$, for $\alpha_1\in\Psi^{red,+}$.

For each $\alpha_1\in\Psi^{red,+}$, we define a function
$c_{\alpha_1}:\hat S\to\ring{C}$ by
\[
c_{\alpha_1}(s) = d_{\alpha_1}(q^{-1})/d_{\alpha_1}(1),
\]
following the definition of $d_{\alpha_1}$ in Equation \ref{eqn:d}.
We define a function $\gamma:\hat S\to \ring{C}$ by
\[
\gamma(s)  = \prod_{\alpha_1\in \Psi^{red,+}} c_{\alpha_1}(s^{-1}).
\]
It follows from Lemma \ref{lemma:prod} that
\begin{equation}\label{eqn:gamma}
\gamma = P(E^{-1},1)/P(E^{-1},q^{-1}).
\end{equation}

The traditional Macdonald formula is expressed as a sum of $\gamma$
over $W^\theta$.  If we use Equation (\ref{eqn:gamma}) to substitute
for $\gamma$ in the traditional formula, then Theorem
\ref{thm:macdonald} is the result.

The formula for $Q_S(q^{-1})$ relies on the observation that $q^{\ell'
  w} = \op{card}(IwI/I)$, where $I$ is an Iwahori subgroup, and the
length is computed with respect to the absolute Weyl group $W$
\cite[p.74]{carter1985finite}.
\end{proof}

\subsection{Plancherel measure}
We continue in the same context, letting $G$ be an unramified
reductive group and $K$ a hyperspecial maximal compact subgroup.  Let
${}^LG = \hat G\rtimes \ang{\theta}$, and we continue with notation
from previous sections.

Let ${\hat S}_1$ be the maximal compact subgroup of $\hat S$.  Let $ds$
be the Haar measure on $\hat S_1$ normalized so that $\hat S_1$ has
volume $1$.  Let $(\cdot,\cdot)$ be the inner product with respect to
the Haar measure on $\hat S_1$. That is,
\begin{equation}
(f_1,f_2) = \int_{\hat S_1} f_1(s) \bar f_2(s) ds.
\end{equation}
Multiplicative characters of $\hat S_1$ are orthonormal:
$(e^\lambda,e^\mu) = \delta_{\lambda,\mu}$.

We define a measure on $\hat S_1$ by
\begin{equation}
dm(s) = \frac{Q(q^{-1})}{\op{card}(W^\theta)}
\frac{P(E,q^{-1}) P(E^{-1},q^{-1})}{P(E,1) P(E^{-1},1)} ds.
\end{equation}
It will be checked below that the partition functions defining $dm(s)$
have nonzero denominators on $\hat S_1$.  Let $\ang{\cdot,\cdot}$ be
the pairing provided by this measure on continuous functions on $\hat
S_1$.  That is,
\begin{equation}
\ang{f_1,f_2} = \int_{\hat S_1} f_1(s) \bar f_2(s) dm(s).
\end{equation}

The proof of the Plancherel measure uses the following averaging lemma.

\begin{lemma}\label{lemma:average} 
  Let $f$ be continuous on $\hat S_1$ and $W^\theta$-invariant.  Then
\[
\ang{f,\hat f_\mu} = c_\mu  (f, e^\mu\frac{ P(E,q^{-1})}{P(E,1)}),
\quad\text{where }\quad
c_\mu = q^{\ang{\mu,\rho^\vee}}\frac{ Q(q^{-1})}{Q_{S(\mu)}(q^{-1})}.
\]
\end{lemma}

\begin{proof}  
  The measure $dm(s)$ is $W^\theta$-invariant by Corollary
  \ref{cor:weyl-p}.  As a consequence, we may push a sum over
  $W^\theta$ over to $f$ to obtain
\begin{equation}\label{eqn:inv}
\ang{f,\sum_{w\in W^\theta} w(f_1)} = \card(W^\theta)\ang{f,f_1},
\end{equation}
for any continuous function $f_1$ on $\hat S_1$.
In particular, assume that $f_1=\hat f_\mu$, which Macdonald's
formula presents a sum over $W^\theta$.
This means that in $\ang{f,\hat f_\mu}$, we may replace 
$\hat f_\mu$  with
the $w=1$ term  in Macdonald's formula.

The constant $c_\mu$ is the product of the constants appearing in
Macdonald's formula, the Plancherel measure, and Equation
\ref{eqn:inv}:
\[
c_\mu = \left(\frac{q^{\ang{\mu,\rho^\vee}}}{Q_S(q^{-1})}\right)
\left(\frac{Q(q^{-1})}{\op{card}(W^\theta)}\right) \op{card}(W^\theta).
\]

The conjugate of $e^\mu$ is $e^{-\mu}$ and of $E$ is $E^{-1}$ on $\hat
S_1$, because $\bar s = s^{-1}$, for $s\in \hat S_1$.  We have
\begin{align*}
\ang{f,\hat f_\mu} &=
c_\mu\int_{\hat S_1} f \left(e^{-\mu}
                     \frac{P(E,1)}{P(E,q^{-1})}\right) 
\left(\frac{P(E,q^{-1}) P(E^{-1},q^{-1})}{P(E,1)(E^{-1},1)}\right) ds\\
&=
c_\mu\int_{\hat S_1} f \left(e^{-\mu} 
\frac{P(E^{-1},q^{-1})}{P(E^{-1},1)}\right) ds\\
&=
c_\mu (f, e^\mu\frac{ P(E,q^{-1})}{P(E,1)}).
\end{align*}
\end{proof}

\begin{theorem}[Plancherel measure]
  The denominators are nonzero on $\hat S_1$ in the partition
  functions defining $dm(s)$.  The Plancherel measure is supported on
  $\hat S_1$ and is given explicitly by $dm(s)$ on $\hat S_1$.
\end{theorem}

\begin{remark}
  We recall the defining property of the Plancherel measure for the
  spherical Hecke algebra.  The Plancherel measure is $dm(s)$ if for
  all $f_1,f_2\in \H(G//K)$,
\[
\int_G f_1(g) \bar f_2 (g) dg = \int_{\hat S} \hat f_1(s) \bar {\hat f}_2 (s) dm(s).
\]
When $\mu\ne\lambda$, the integral on the left is trivial to compute
for the functions $f_1 = f_\lambda$ and $f_2 = f_\mu$ because the
functions $f_\mu$ and $f_\lambda$ have disjoint support, giving $0$.
When $\mu=\lambda$, the integral on the left is the volume of
$K\varpi^\mu K$.  This volume is $c_\mu q^{\ang{\mu,\rho^\vee}}$ by
\cite{casselman2005companion}, where $c_\mu$ is the constant in Lemma
\ref{lemma:average}.  The proof of the theorem proceeds by computing
the inner products $\ang{\hat f_\lambda,\hat f_\mu}$ and showing that
they give the same values as the integral on the left.
\end{remark}

\begin{proof}  
  The proof, which we review, is due to Macdonald
  \cite[Ch.V]{macdonaldspherical}.  It is what he calls {\it the
    standard case}.  Choose any total order $(<)$ on $P^+\subset Y^*$
  such that $\lambda < \lambda + \alpha$, whenever $\alpha = N\beta$
  is the norm of a positive root $\beta\in\Psi^+$.

  By Lemma \ref{lemma:prod}, the ratio $P(E,q^{-1})/P(E,1)$ factors
  into a product of terms of the form $(1- t)/(1- q^{-b} t)$, where
  each $t = \epsilon e^{N\beta} = \epsilon e^\alpha$ for some root
  $\beta$ with norm $\alpha$, for some sign $\epsilon\in \{\pm 1\}$,
  and for some $b\ge 1$.  For any $p$-adic field $F$, we have $q = q_F
  > 1$ and $q^{-b} < 1$.  Thus we have an absolutely convergent
  expansion in $t$:
\begin{equation}\label{eqn:t}
\frac{1- t}{1- q^{-b} t} = 1 + (q^{-b}-1) t (1+ q^{-b} t + q^{-2b} t^2 + \cdots),
\end{equation}
noting that $|t| = |\alpha(s)|=1$ at each $s\in \hat S_1$.  In
particular, the denominator of $P(E,q^{-1})/P(E,1)$ does not vanish.
Similarly, the denominator of $P(E^{-1},q^{-1})/P(E^{-1},1)$ does not
vanish because $|q^{-b} t^{-1}|\ne1$, giving the nonvanishing of the
denominator in the measure $dm(s)$.

By multiplying the series expansions (Equation \ref{eqn:t}) associated
with each factor of $P(E,q^{-1})/P(E,1)$, it follows that for each
$\mu\in P^+$ we have an absolutely convergent expansion of the form
\[
e^\mu \frac{P(E,q^{-1})}{P(E,1)} = e^\mu +\sum_{\mu< \mu'} a_{\mu'} e^{\mu'},
\] 
for some coefficients $a_{\mu'}$ that turn out not to matter.

We compute $\ang{\hat f_\lambda,\hat f_\mu}$.  We may assume without
loss of generality that $\lambda \le \mu$.

We have a finite expansion (see Equation \ref{eqn:s} below):
\[
\hat f_\lambda = q^{\ang{\lambda,\rho^\vee}} m_\lambda 
+ \sum_{\lambda' <\lambda } s_{\lambda,\lambda'} m_{\lambda'}.
\]
Also, for $\lambda',\mu'\in P^+$, we have
\[
(m_{\lambda'},e^{\mu'}) = \delta_{\lambda',{\mu'}}.
\]
The function $\hat f_\lambda$ is $W^\theta$-invariant, which justifies
the use of the averaging lemma (Lemma \ref{lemma:average}) to simplify
the inner product.  Invoking the averaging lemma, expanding
everything, and integrating term by term, we have
\begin{align*}
\ang{\hat f_\lambda,\hat f_\mu} 
&= c_\mu (\hat f_\lambda, e^\mu \frac{ P(E,q^{-1})}{P(E,1)})\\
&= c_\mu (q^{\ang{\lambda,\rho^\vee}} m_\lambda,e^\mu) 
+ \sum_{\lambda' < \lambda\le\mu < \mu'}
c_\mu s_{\lambda',\lambda} a_{\mu'} (m_{\lambda'},e^{\mu'})\\ 
&= c_\mu q^{\ang{\mu,\rho^\vee}} \delta_{\lambda,\mu}.
\end{align*}

Comparing this inner product with the inner products in the remark, we
see that the proof is complete.
\end{proof}

\subsection{inverting weight multiplicities}

This section follows van Leeuwen's algorithm to invert the weight
multiplicity matrix~\cite{vanleeuwen}.  For type $A_n$, van Leeuwen's
formula agrees with the inverse of the Kostka matrix described in
\cite{duan}.

We have two bases of $\ring{C}[Y^*]^{W^\theta}$, given by
$\{m_\lambda\}$ and $\{\tau_\lambda\}$, indexed by $\lambda\in \dom$.
The change of basis matrix expressing $\tau_\lambda$ in terms of
$m_\mu$ is the weight multiplicity matrix $m_{\lambda,\mu} =
(\tau_\lambda,e^\mu)$.  In the reverse direction, for $\mu\in P^+$, we
have a change of basis matrix $n_{\mu,\lambda}$
\begin{equation}\label{eqn:n}
m_\mu = \sum_{\lambda} n_{\mu,\lambda} \tau_\lambda,
\end{equation}
with $\mu,\lambda\in P^+$.
This section gives a formula for $n_{\mu,\lambda}$.  

We have a set $Y^*_0\subseteq Y^*$ of characters $\lambda$ such that
$\lambda$ is fixed (that is, $w\bullet \lambda = \lambda$) by some
reflection $w\in W^\theta$.  For each $w\in W^\theta$, we define
\begin{equation}
Y^*_w = \{\lambda\in Y^*\mid w\bullet \lambda \in\dom\}.
\end{equation}
These sets partition $Y^*$, so that each $\lambda\in Y^*$ belongs to a
unique $Y^*_x$, for $x\in W^\theta\cup\{0\}$.  Let $e_w$ be the
characteristic function of $Y^*_w$.

Recall (from Section \ref{sec:weyl-char}) that we have defined an
operator $J$ that has the property $J(f) = f J(1)$ if $f\in
\ring{C}[Y^*]^{W^\theta}$.  In particular, using the Weyl denominator
formula (\ref{eqn:wd2}), we find that
\begin{equation}
 J(\tau_\lambda) = \tau_\lambda J(1) 
= J(e^\lambda) P(E^{-1},1) J(1) = J(e^\lambda),\quad  \lambda\in P^+.
\end{equation}
In the opposite direction, we define a desymmetrizer operator $L$ by
\[
L(e^\mu) = \begin{cases}
0,& \mu\in Y^*_0;\\
(-1)^{\ell{w}} e^{w\bullet \mu},& \mu\in Y^*_w.
\end{cases}
\]
We extend $L$ linearly to $\ring{C}[Y^*]$.  The operator $L$ can be
characterized as the unique linear operator whose range is supported
on the dominant weights and such that $J(f) = J(L(f))$, for $f\in
\ring{C}[Y^*]$.  By this characterization, 
\[
L(\tau_\lambda) = L(J(e^\lambda)) = L(e^\lambda) =
e^\lambda, 
\]  
for $\lambda\in\dom$.
This means that for any $f =
\sum_{\lambda\in P^+} c_\lambda \tau_\lambda \in \ring{C}[Y^*]$, the
coefficient $c_\lambda$ is the coefficient of $e^\lambda$ in $L(f)$.
This is {\it van Leeuwen inversion}.
We have
\[
L(e^\mu) = \sum_{w\in W^\theta} (-1)^{\ell w} e^{w\bullet \mu} e_w(\mu).
\]

Recall that $S(\mu)$ is defined as the set of simple roots such that
$\ang{\alpha^\vee,\mu}=0$ iff $\alpha\in S(\mu)$.  Recall also that
$W_S\le W^\theta$ is the subgroup generated by reflections in $S$.

\begin{lemma}[van Leeuwen]  \label{lemma:van}
  For each subset $S$ of the set of simple roots of $\Psi^{red,+}$, and
  for every $\mu\in P^+$ with $S = S(\mu)$,
\[
n_{\mu,\lambda}=\sum_{(w',w)\in (W^\theta/W_S)\times W^\theta}
 (-1)^{\ell(w)} e_w(w'\mu) \delta_{w\bullet (w'\mu),\lambda}. 
\]
\end{lemma}

\begin{proof}  
\[
m_\mu = \sum_{w'\in W^\theta/W_S} e^{w' \mu}.
\]
Then
\begin{align*}
n_{\mu,\lambda} 
    &= (\sum_{\lambda'} n_{\mu,{\lambda'}} e^{\lambda'},e^\lambda) \\
     &= (L(\sum_{\lambda'} n_{\mu,\lambda'} \tau_{\lambda'}),e^\lambda) \\
     &= (L(m_\mu),e^\lambda) \\
     &= \sum_{w'\in W^\theta/W_S} (L(e^{w'\mu}),e^\lambda)\\
     &= \sum_{w'\in W^\theta/W_S} \sum_{w\in W^\theta} 
      (-1)^{\ell w} e_w(w'\mu) (e^{w\bullet (w'\mu)},e^\lambda)
\end{align*}
\end{proof}

\subsection{geometric Satake}

Let $K$ be a hyperspecial maximal compact subgroup of $G$.  In the
usual formulation, the Satake transform is an isomorphism of the Hecke
algebra $\H(G//K)$, with the $W^\theta$-invariant functions
in the group algebra $\ring{C}[X_*(A)]$.  As first pointed out by
Langlands, this is the space of conjugation invariant functions on the
coset $\hat G\rtimes \theta$ in ${}^LG$.  For split groups, the
geometric Satake transform reformulates the transform in the language
of sheaves and in terms of the representation ring of $\hat G$
\cite{mirkovic2007geometric}.  For a geometric approach to geometric
Satake that includes unramified groups, see \cite{zhu2011geometric}.

The identities we give should be viewed as formal analogues of
geometric Satake by the function-sheaf
dictionary. 
Working formally with irreducible characters, the geometric Satake
transform expresses each $\hat f_\lambda$ in terms of the basis
$\tau_\mu$ of irreducible characters for the root system $\Psi^{red}$:
\begin{equation}\label{eqn:geometric-satake}
\hat f_\lambda = \sum_\mu g_{\lambda,\mu} \tau_\mu.
\end{equation}
Casselman gives a formula for the coefficients $g_{\lambda,\mu}$ when
$G$ is split \cite{casymmetric}.  In this section, we extend the
result to $G$ unramified.  The statement here is less polished than
what is known in the split case \cite{casymmetric}.  The proof we
present here is based on van Leeuwen's algorithm.

Set
\[
C = \{  - \sum_{\alpha\in S}\alpha \ \ \mid S \subseteq \Psi^+ \}^\theta;
\]
Define $p_\mu(q^{-1})$ by
\begin{equation}\label{eqn:p}
 P(E^{-1},q^{-1})^{-1} = \sum_{\mu\in C} p_\mu(q^{-1}) e^{\mu},
\text{\ \ so }  P(E^{-w},q^{-1})^{-1} 
= \sum_{\mu\in C} p_\mu(q^{-1}) e^{w\mu}.
\end{equation}

\begin{theorem}[geometric Satake]\label{thm:gs}  
Let $\lambda\in P^+$ and let $S=S(\lambda)$.
\[
\hat f_\lambda = \frac{q^{\ang{\lambda,\rho^\vee}}}{Q_S(q^{-1})} 
\sum_{\mu\in C}\left(\sum_{w\in W^\theta} 
(-1)^{\ell w} p_{\mu}(q^{-1}) e_w(\lambda+\mu)\right)  
\tau_{w\bullet(\lambda+\mu)}.
\]
\end{theorem}

\begin{proof}
  Abbreviate $r_\lambda = q^{\ang{\lambda,\rho^\vee}}/Q_S(q^{-1})$.
  Let $\rho=\rho(\Psi^+)$.
  By the Weyl denominator product formula (Corollary \ref{cor:prod1}),
\begin{equation}
P(E^{-w},1) = (-1)^{\ell w} e^{w\rho - \rho} P(E^{-1},1),
\end{equation}
Then expanding Macdonald's formula (\ref{thm:macdonald}) using this
and Equation \ref{eqn:p}, we get
\[
\hat f_\lambda = r_\lambda \sum_{\mu\in C} 
p_\mu(q^{-1}) P(E^{-1},1) J (e^{\lambda+\mu}),
\]
where we have absorbed the sum over $W^\theta$ in Macdonald's formula
into $J$.  We observe that
\begin{align*}
L(\hat f_\lambda) &= r_\lambda 
\sum_{\mu\in C} p_\mu(q^{-1}) L(P(E^{-1},1) J(e^{\lambda+\mu})) \\
&= 
r_\lambda \sum_{\mu\in C} p_\mu(q^{-1}) L(e^{\lambda+\mu}) \\
&= 
r_\lambda
\sum_{\mu\in C} p_\mu(q^{-1}) 
\sum_{w\in W^\theta} (-1)^{\ell w} e_w(\lambda+\mu) e^{w\bullet (\lambda+\mu)}.
\end{align*}
Recall that the coefficient $c_\lambda$ of an expansion $f =
\sum_\lambda c_\lambda\tau_\lambda$ is the coefficient of $e^\lambda$
in $L(f)$.  The result follows.
\end{proof}

There is a second less explicit form of the geometric Satake transform
that is obtained as follows.  We have
\[
\hat f_\lambda = \sum_{\mu'} s_{\lambda,\mu'} m_{\mu'},
\]
where the coefficients $s_{\lambda,\mu'}$ are given as $p$-adic
integrals (see Equation \ref{eqn:s} below).  By van Leeuwen's formula
linking $m_{\mu'}$ to $\tau_\mu$ with coefficient matrix
$n_{\mu',\mu}$, we obtain the coefficients $g_{\lambda,\mu}$ as a
matrix product:
\begin{equation}
\hat f_\lambda 
= \sum_{\mu',\mu} s_{\lambda,\mu'} n_{\mu',\mu}
\tau_{\mu} 
= \sum_{\mu} g_{\lambda,\mu} \tau_\mu.
\end{equation}

\subsection{a Kato-Lusztig formula}

Equipped with Plancherel and Macdonald, we obtain an easy Kato-Lusztig
formula for the inverse Satake transform.  Our result generalizes a
formula that was known when $\theta=1$ \cite{kato1982spherical}
\cite{lusztig1983singularities}.  Recall the $q$-twisted character
$\tau_{\lambda,q}$ from (\ref{eqn:tauq}).  Write
\[
\tau_\lambda = \sum_\mu t_{\lambda,\mu}  \hat f_\mu,
\]
for some constants $t_{\lambda,\mu}$.

\begin{theorem}[Kato-Lusztig formula]
  The coefficients $t_{\lambda,\mu}$ of the inverse geometric Satake
  transform are
\[
t_{\lambda,\mu} =  (\tau_{\lambda,q^{-1}},e^\mu) q^{-\ang{\mu,\rho^\vee}}.
\]
\end{theorem}

\begin{proof}
  The character $\tau_\lambda$ is $W^\theta$-invariant.  We use the
  averaging property (Lemma \ref{lemma:average}) and the Weyl
  character formula to compute an inner product.
\begin{align*}
\ang{\tau_\lambda,\hat f_\mu}
&=c_\mu(\tau_\lambda,e^\mu \frac{P(E,q^{-1})}{P(E,1)})\\
&=c_\mu\int_{\hat S_1} \left(J(e^\lambda) P(E^{-1},1)\right) 
\left( e^{-\mu}\frac{P(E^{-1},q^{-1})}{P(E^{-1},1)} \right) ds\\
&=
c_\mu\int_{\hat S_1} J(e^\lambda) P(E^{-1},q^{-1}) e^{-\mu} ds\\
&= c_\mu(\tau_{\lambda,q^{-1}},e^\mu).\\
\ang{\tau_\lambda,\hat f_\mu}
&=\sum_{\mu'} t_{\lambda,\mu'} \ang{\hat f_{\mu'},\hat f_\mu} \\
&= t_{\lambda,\mu} c_\mu q^{\ang{\mu,\rho^\vee}}.
\end{align*}
\end{proof}

\section{Endoscopic branching rules}\label{sec:branch}

This section uses partition functions to give a branching rule for
the restriction of an irreducible representation of $\hat G\rtimes\theta$
to $\xi(\hat H\rtimes \theta_H)$.

\subsection{$\theta$-conjugacy}\label{sec:theta-conj}

Let $G$ be an unramified reductive group, and let ${}^LG = \hat G
\rtimes \ang{\theta}$ be its $L$-group, with the automorphism $\theta$
given by the action of the Frobenius element on the root datum.  
The calculations in this subsection will be used in Theorem
\ref{thm:branch} to give an explicit branching rule for an embedding of
the $L$-group of an endoscopic groups into the $L$-group of $G$.

We have a set of simple roots $\Delta\subseteq \Psi$, determined by
$(\hat T,\hat B)$.  Let $\alpha$ be the highest positive root and let
$\Delta^e = \Delta \cup \{-\alpha\}$ be the extended set of simple
roots.  The extended Dynkin diagram has node set $\Delta^e$.  The
automorphism $\theta$ preserves $\Delta$ and fixes $-\alpha$, hence
acts on the extended Dynkin diagram.

Let $w\in W$ be an element that preserves the extended Dynkin diagram.
We consider lifts $\dotw\in N_{\hat G}(\hat T)$ of $w$ such that
$\theta_1 = \dotw \theta$ has finite order.  The partition function
and other data we define are sensitive to the
representative $\dotw$ of $w$.  However, the branching rule that we
obtain in the end (Theorem \ref{thm:branch}) will satisfy a simple
transformation rule depending on $\dotw$ (Lemma
\ref{lemma:transform}).  In each case, we pick a particularly
convenient representative $\dotw$ of $w$ to work with, and leave the
rest to the transformation rule.  The details of the choice of $\dotw$
will be discussed further below.

Let $\hat S = \hat T/(1-\theta)\hat T$ and $\hat T_1 = \hat
T/(1-\theta_1)\hat T$.  There are norm maps $N:X^*(\hat T)\to X^*(\hat
S)$ and $N_1:X^*(\hat T)\to X^*(\hat T_1)$.  

We write $\Psi^+(\hat T,\hat B)$ for the set of positive roots of
$\hat G$ with respect to a Cartan subgroup $\hat T$ and a Borel subgroup
$\hat B$.

\begin{lemma}\label{lemma:adapted}
  There exists a Borel subgroup $\hat
  B(w\theta)\supseteq \hat T$ such that for every
  $\alpha\in\Psi^+(\hat T,\hat B(w\theta))$, either $N_1\alpha = 0$
  or $\ang{w\theta}\alpha\subseteq \Psi^+(\hat T,\hat B(w\theta))$.
\end{lemma}

We call $\hat B(w\theta)$ an {\it adapted} Borel subgroup.

\begin{proof}  Let 
\[
C = \{N_1\alpha\in X^*(\hat T)^{\theta_1} \mid \alpha\in\Psi,\ N_1\alpha\ne 0\}
\subseteq X^*(\hat T_1).
\]
Clearly $N_1(-\alpha)= - N_1\alpha$ and $-C = C$.  Choose a hyperplane
through the origin in $X^*(\hat T_1)\otimes\ring{Q}$ that does not
meet $C$ to partition $C = C_+ \sqcup C_-$ into a positive and
negative set.  We can choose a compatible hyperplane through the
origin in $X^*(\hat T)\otimes\ring{Q}$ such that $\alpha$ is positive
or negative according to $N_1\alpha\in C_\pm$, provided that
$N_1\alpha\ne 0$.  By a small generic perturbation of this hyperplane
through the origin, we may assume that the hyperplane separates each
pair $\pm \alpha$ of roots. 
Let $\hat B(w\theta)$ be the Borel subgroup defined by $\hat T$ and
the positive roots determined by the hyperplane.
\end{proof}

We consider data $\D=(\hat U,\hat B(w\theta),\hat
B_1,\iota,\phi,\epsilon,\dotw)$ of the following type: $\hat U = \hat
T_1$ and $\iota:\hat U\to \hat T_1$ is an isogeny of tori.  Also
$\phi:\hat U\to \hat S$ is a homomorphism, and $\epsilon\in \hat S$ is
an element of finite order.  Finally, $\hat B(w\theta)$ is an
adapted Borel subgroup and $\hat B_1$ is a $\theta$-stable
Borel subgroup of $\hat G$ containing $\hat T$.

The purpose of the isogeny is to remove all radicals from the formulas
that follow.  We always have $\hat U = \hat T_1$, but we maintain two
notations for the same torus to distinguish the source $\hat U$ from
the target $\hat T_1$ of the isogeny $\iota$.  We write $\uu$ for an
element of $\hat U$ and $\iota(\uu)=\tau\in \hat T_1$ for an element
of the target of the isogeny.  Define $\phi_\epsilon:\hat U\to\hat S$
by $\phi_\epsilon(\uu)= \epsilon\phi(\uu)$.

As noted above, $D(\hat G,R,\theta,E,q)$ is a function on $\hat S$,
which we pull back to a function $\phi^*_\epsilon D(\hat
G,R,\theta,E,q)$ on $\hat U$.  Similarly, $\iota^* D(\hat G,R',\dotw
\theta,E,q)$ is a function on $\hat U$.  The $\hat G$-conjugacy class
of $t\dotw \theta$ depends only on the image $\tau\in \hat T_1$ of
$t\in \hat T$.  We can therefore refer to the $\hat G$-conjugacy class
of $\tau\dotw\theta$, for $\tau\in \hat T_1$.

Let $\Psi_{\phi,\theta}^+$ be the set of roots $\alpha$ of
  $\Psi^+(\hat T,\hat B_1)$ such that $\phi^*N\alpha\ne 0$, and let
  $\Psi_{w\theta}^+$ 
be the set of roots $\alpha$ of $\Psi^+(\hat T,\hat B(w\theta))$ such
  that $N_1\alpha\ne0$.  

\begin{proposition}\label{lemma:ephi}
  Let $G$ be an unramified reductive group with complex dual $\hat
  G\rtimes \ang{\theta}$.  Assume that $w\theta$ acts on the extended
  root system.  Then we can construct data $\D=(\hat U,\hat
  B(w\theta),\hat B_1,\iota,\phi,\epsilon,\dotw)$ typed as above such
  that
\begin{enumerate}
\item (conjugacy) $\tau\dotw\theta$ is $\hat G$-conjugate to
  $\epsilon\phi(\uu) \theta$, where $\iota(\uu) = \tau\in\hat T_1$;
\item (regularity) For each $\alpha\in\Psi^{red}$, the zero set of
  $\alpha(\phi_\epsilon(\uu))-1$ is a proper Zariski closed subset of
  $\hat U$.
\item (partition) 
\[
\phi^*_\epsilon D(\hat G,\Psi_{\phi,\theta}^+,\theta,E,q) =
\iota^* D(\hat G,\Psi_{w\theta}^+,\dotw\theta,E,q).
\]
\end{enumerate}
\end{proposition}

\begin{remark}
  We note that $\hat T_1/W_H^{\theta_H}$ classifies
  $\theta_H$-semisimple conjugacy classes in ${}^LH$, and $\hat
  S/W_G^{\theta}$ classifies $\theta$-semisimple conjugacy classes in
  ${}^LG$.  A morphism $\xi:{}^LH\to {}^LG$ induces a map of conjugacy
  classes $\hat T_1/W_H^{\theta_H}\to \hat S/W_G^\theta$.  The
  correspondence
\[
\hat T_1 \longleftarrow \hat U \longrightarrow \hat S
\]
between $\hat T_1$ and $\hat S$ lifts this map of conjugacy classes up
to tori.  The identity of partition functions refines the identity
between the characteristic polynomials of conjugate elements.

The morphism $\xi$ gives a restriction map
\begin{equation}\label{eqn:Jxi}
J_\xi:\ring{C}[Y^*]^{W^\theta} \to \ring{C}[X^*(\hat T_1)]^{W_H^{\theta_H}}
\end{equation}
that is constructed as follows.  The morphism $\hat
T_1/W_H^{\theta_H}\to \hat S/W_G^\theta$ is a morphism of affine
varieties, which determines a homomorphism $J_\xi$ between their
coordinate rings.
\end{remark}

\begin{remark}
  We have calculated the determinants on a case-by-case basis and the
  answers can be striking.  For example, assume that $\theta=1$ and
  $\hat G$ is simply laced.  Let
  $N_1\Psi=\{N_1\beta\mid \beta\in \Psi\}$ be the norm root system of
  $\hat T_1$.  Let $m:\Psi\to \ring{N}$ be given by $m(\alpha)=$ the
  order of $w$, if $\alpha$ is a long root in $N_1\Psi$ (or if
  $N_1\Psi$ is simply laced).  Let $m(\alpha) = 1$, otherwise.  Then
  there exist data $\D$ such that
\begin{equation}\label{eqn:laced0}
D(\hat G,\Psi_{w\theta}^+,\dotw,E,q) 
= \prod_{\alpha\in N_1\Psi_{w\theta}^{+}} 
(1- q^{m(\alpha)} e^\alpha)^{m(\alpha)}.
\end{equation}
\end{remark}

\begin{proof}[Proof (proposition)]

  The proof extends over several subsections giving a series of
  reductions.  The first reduction is Levi descent.

\subsubsection{Levi descent}

Let $\hat P$ be a $\theta$-stable parabolic subgroup containing $\hat
T$.  Let $\hat M$ be a $\theta$-stable Levi subgroup of $\hat P$
containing $\hat T$.  Let $\hat M_{sc}$ be the simply connected cover
of the derived group of $\hat M$.  Let $w\in W(\hat T,\hat M) = W(\hat
T_{sc},\hat M_{sc})$ be a Weyl group element that acts on the extended
Dynkin diagrams of both $\hat M_{sc}$ and $\hat G$ (for some choices
of $\theta$-stable Borel subgroups). Then $w\theta$ is an automorphism
of both $\hat M_{sc}$ and $\hat G$.  We assume that we have
constructed data $\D_{sc} = (\hat U_{sc},\ldots)$ for $\hat M_{sc}$
and show how to construct data $\D$ for $\hat G$.

Let $\hat A = Z(\hat M)^0$ be the identity component of the center of
$\hat M$.  Set
  \[
  \hat A_\theta :=  \hat A/(1-\theta)\hat A = \hat A / (1-w\theta)\hat A.
  \]
  Define $\hat U$, $\phi:\hat U\to \hat S$, and $\iota:\hat U\to \hat T_1$ 
  as follows:
  \begin{align*}
  \phi:\hat U &:= \hat A_\theta \times \hat U_{sc} 
  \to \hat A_\theta \times \hat S_{sc}\to \hat S.
  \\
  \iota:\hat U &= \hat A_\theta \times \hat U_{sc} 
  \to \hat A/(1-w\theta)\hat A\,\times
    \hat T_{1,sc} \to \hat T_1.
  \end{align*}
  Let $\Psi_N$ be the set of positive roots of the unipotent radical
  of $\hat P$.  Let $\hat B(w\theta)$ be the adapted Borel containing
  $\hat T$ with positive roots those of $\hat
  B(w\theta)_{sc}\subseteq\hat M_{sc}$ and $\Psi_N$.  Then
  \[
  \Psi_{w\theta}^+ = \Psi_{w\theta,sc}^+ \sqcup \Psi_N.
  \]
  Let $\hat B_1$ be the Borel containing $\hat T$ with positive roots
  those of $\hat B_{1,sc}\subseteq\hat M_{sc}$ and $\Psi_N$.

  We claim that $\Psi_{\phi,\theta}^+ = \Psi_{\phi\,\theta,sc}^+\sqcup
  \Psi_N$.  Note that
  \[
  \Psi_{\phi,\theta}^+ = (\Psi_{\phi,\theta}^+\cap \Psi^+(\hat T,\hat B_{1,sc})) \sqcup
(\Psi_{\phi,\theta}^+ \cap \Psi_N) = 
\Psi_{\phi,\theta,sc}^+\sqcup (\Psi_{\phi,\theta}^+ \cap \Psi_N).
\]
We have a map
  \[
  \hat A^{\theta,0} \to  \hat A_\theta \to \hat U \to \hat S.
  \]
  Also, $\alpha$ is trivial on $\hat A^{\theta,0}$ iff it is trivial on $\hat A$
  iff it is a root of $\hat M$.  For all $\alpha\in \Psi_N$,
  \[
  \phi^* N\alpha = 0 \ \Leftrightarrow \ 
  N\alpha|_{\hat A^{\theta,0}}=0 \ \Leftrightarrow \ 
  k\alpha|_{\hat A^{\theta,0}}=0 \ \Leftrightarrow \ 
  \alpha|_{\hat A}=0.
  \]
  The claim follows.

  We choose $\epsilon\in \hat S$ to be the image of $\epsilon_{sc}$
  and $\dotw$ to be the image of $\dotw_{sc}$.

  We prove property (1-conjugacy).  Let $(a,\uu_{sc})\in \hat A\times
  \hat U_{sc}$ represent $\uu\in \hat U_{sc}$.  The element
  $\epsilon_{sc}\phi_{sc}(\uu_{sc}) \theta$ is conjugate to
  $\iota_{sc}(\uu_{sc})\dotw_{sc}\theta$ by an element $m\in\hat
  M_{sc}$ whose image in $\hat M$ commutes with $a\in Z(\hat M)$.
  Thus $a \epsilon_{sc}\phi_{sc}(\uu_{sc}) \theta$ and
  $a\iota_{sc}(\uu_{sc})\dotw\theta$ are conjugate.

  We show that property (2-regularity) of the proposition holds.  As
  shown above, if $\alpha$ is not a root of $\hat M$, then
  $\phi^*N\alpha\ne 0$, so the regularity property holds for
  $\alpha\in\Psi_N$.  For the remaining positive roots, regularity
  follows from regularity for $\hat M_{sc}$.

  Finally we show (3-partition).  Let $m$ and $(a,\uu_{sc})$ be as
  above.  Let $\g_N = \g_{\Psi_N}$.  We have an isomorphism
  $\op{Ad}(m^{-1}):\g_N\to \g_N$.  Using the conjugacy property, we
  have
    \begin{align*}
      \phi_\epsilon^* D(\hat G,\Psi_N,\theta,E,q)(\uu) &=
      \det(1-\epsilon \phi_{sc}(\uu_{sc}) a \theta q;\g_N) \\
      &=
      \det(1-  \iota_{sc}(\uu_{sc})  a \dotw\theta q;\g_N) \\
      &=
      \iota^* D(\hat G,\Psi_N,\dotw\theta,E,q)(\uu).
      \end{align*}
Therefore,
\begin{align*}
\phi_\epsilon^* D(\hat G,\Psi_{\phi,\theta}^+,\theta,E,q) &=
\phi_\epsilon^* D(\hat G,\Psi_{\phi,\theta,sc}^+,\theta,E,q)
\phi_\epsilon^* D(\hat G,\Psi_N,\theta,E,q) \\
&=
\phi_{\epsilon,sc}^* D(\hat M,\Psi_{\phi,\theta,sc}^+,\theta,E,q)
\iota^* D(\hat G,\Psi_N,\dotw\theta,E,q)\\
&=
\iota_{sc}^* D(\hat M,\Psi_{w\theta,sc}^+,\dotw_{sc}\theta,E,q)
\iota^* D(\hat G,\Psi_N,\dotw\theta,E,q)\\
&=
\iota^* D(\hat G,\Psi_{w\theta}^+,\dotw\theta,E,q).
\end{align*}
This completes the reduction to $\hat M_{sc}$.    

As a special case of this construction, applied to $\hat M=\hat G$, we
reduce to the case that $\hat G$ is semisimple and simply connected.

\subsubsection{reduction to transitive on orbits}

The Dynkin diagram of $\hat G$ is a union of orbits under $\theta$.
We give a reduction to the case that $\theta$ is transitive on the set
of connected components of the Dynkin diagram. We may assume that
$\hat G$ is simply connected.  Assume that data $\D_i=(\hat
U_i,\ldots)$ satisfying the properties of the proposition have been
constructed from each factor of $\hat G = \hat G_1\times\cdots\times
\hat G_r$, where $\theta=(\theta_1,\ldots,\theta_r)$, and $\theta_i$
acts on $\hat G_i$.  We have factorizations of $\hat T_1$ and $\hat S$
as $r$-fold products.  We may define the data $\D=(\hat U,\ldots)$ for
$\hat G$ as a $r$-tuples $\phi = (\phi_1,\ldots,\phi_r)$, $\epsilon =
(\epsilon_{1},\ldots,\epsilon_{r})$, and $r$-fold products, etc.  The
verification of properties (1), (2), (3) is routine in this case.
This completes the reduction.

\subsubsection{$W$-conjugate reduction}\label{sec:conjred}

In this subsection, we show how to construct data $\D$ for $w\theta$
assuming that we have data $\D'$ for $w'\theta$, where $w_1w\theta
w_1^{-1}=w'\theta$ for some $w_1\in W$ and assuming that both $w$ and
$w'$ act on the extended Dynkin diagram (for some common choice of
simple roots).

Assume that $\D' =(\hat U',\ldots,\dotw')$ is given.  Let $\dotw_0$
and $\dotw_1$ be lifts of $w$ and $w_1$ to $N_{\hat G}(\hat T)$.  Then
\[
\dotw_1 \dotw_0 \theta\dotw_1^{-1} = t \dotw'\theta,
\]
for some $t\in \hat T$.  Let $\dotw = \dotw_1^{-1} t^{-1}
\dotw_1\dotw_0$.  Then $\dotw\mapsto w$ and
\begin{equation}\label{eqn:outer}
\dotw_1\dotw \theta\dotw_1^{-1} = \dotw'\theta.
\end{equation}

Set $\hat B(w\theta) = \hat B(w'\theta)^{w_1}$.  Then $\hat
B(w\theta)$ is adapted to $w\theta$.  Also,
\[
\Psi_{w\theta}^+ = w_1^{-1}\Psi_{w'\theta}^+.
\]

We have isomorphisms
\begin{equation}
\hat T_1 = \hat T/(1-w\theta) \hat T 
= \hat T/(1-w' \theta)\hat T  = \hat T_1'.
\end{equation}
The isomorphism $\hat T/(1-w \theta)\hat T\to \hat T/(1-w'\theta)\hat
T$ is given by $t \mapsto w_1 t w_1^{-1}$.  Let $\hat U = \hat U'$.
Define $\iota$:
\[
\iota:\hat U = \hat U' \to^{\iota'} \hat T_1' \to^{\op{Int}(w_1^{-1})} \hat T_1.
\]

Let $\g$ be the Lie algebra of $\hat G$.  We have a linear isomorphism
\[
\psi:\g\to \g,\quad X\mapsto \op{Ad}(\dotw_1) X,
\]
sending $\g_{\Psi_{w\theta}^+}$ to $\g_{\Psi_{w'\theta}^+}$. Then
\begin{align*}
D(\hat G,\Psi_{w\theta}^+,\dotw\theta,E,q) &=
D(\hat G,\dotw_1^{-1}\Psi_{w'\theta}^+,\dotw_1^{-1}\dotw'\theta \dotw_1,E,q) \\ &=
D(\hat G,\Psi_{w'\theta}^+,\dotw'\theta,E^{w_1^{-1}},q).
\end{align*}

On the other hand, choose $\phi=\phi'$, $\hat B_1 = \hat B_1'$,
and $\epsilon=\epsilon'$.
Then  $\Psi_{\phi,\theta}^+ =\Psi_{\phi',\theta}^+$ and
\begin{align*}
i^* D(\hat G,\Psi_{w\theta}^+,\dotw\theta,E,q) &=
i^* D(\hat G,\Psi_{w'\theta}^+,\dotw'\theta,E^{w_1^{-1}},q) \\
&=i^{'*} (\op{Int} (w_1^{-1}))^* 
  D(\hat G,\Psi_{w'\theta}^+,\dotw\theta,E^{w_1^{-1}},q) \\
&=i^{'*}  D(\hat G,\Psi_{w'\theta}^+,\dotw'\theta,E,q) \\
&=\phi_\epsilon^{'*}  D(\hat G,\Psi_{\phi',\theta}^+,\theta,E,q) \\
&=\phi_\epsilon^{*}  D(\hat G,\Psi_{\phi,\theta}^+,\theta,E,q).
\end{align*}
These determinant formulas give property (3) of the proposition.

To prove property (1-conjugacy), we can realize the $\hat G$-conjugacy
of $\epsilon\phi(\uu) \theta$ and $\iota(\uu)\dotw\theta$ explicitly by
conjugation by $\dotw_1$ and the conjugation used for data $\D'$.

To prove property (2-regularity), we observe that
$\alpha(\epsilon\phi(\uu))=1$ iff $\alpha(\epsilon'\phi'(\uu))=1$.
This holds for all $\uu$ iff $\alpha=0$.  This completes the reduction
from $w\theta$ to $w'\theta$.

\subsubsection{reduction to simple}

  We give a further reduction to $\hat G$ simple.  In fact, let
  $\hat G'$ be a factor of $\hat G$.  Let $k$ be the smallest integer
  such that $\theta':=\theta^k$ acts on $\hat G'$.  Assume that we
  have data $\D'=(\hat U',\ldots)$ for $\hat G'$.  We write
  \[
  \hat G = \hat G' \times \theta(\hat G') \times \cdots,
  \quad
  \theta(g_1,\ldots,g_k) = (\theta' g_k, g_1,g_2,\ldots,g_{k-1}).
  \]
  Using Section \ref{sec:conjred} to change $w\theta$ to a conjugate
  element, we may assume that
  \[
  \dotw = (\dotw',1,1,\ldots,1),\quad \D' 
  = (\hat U',\hat B(w'\theta')',\ldots,\dotw').
  \]
  We have $\hat T = \hat T'\times \theta(\hat T')\times \cdots$, $\hat
  S = \hat S'$, $\hat T_1 = \hat T_1'$, $\hat U = \hat U'$,
  $\phi=\phi'$, $\iota=\iota'$, $\epsilon = \epsilon'$.  We can define
  $\hat B_1$ and $\hat B(w\theta)$ by their positive roots, which we
  take to be
  \[
  \{\alpha,\theta\alpha,\ldots,\theta^{k-1}\alpha\mid \cdots\}
  \]
  as $\alpha$ runs over positive roots of $\hat B_1'$ and $\hat B(
  w'\theta')'$, respectively.  Properties (1-conjugacy) and
  (2-regularity) follow from the corresponding properties of $\hat
  G'$.  The determinant formula follows from the identities
  \begin{align*}
    \phi^*_\epsilon D(\hat G,\Psi_{\phi,\theta}^+,\theta,E,q) &=
    \phi'^*_{\epsilon'} D(\hat G',\Psi_{\phi',\theta'}^+,\theta',E',q^k).\\
    \iota^* D(\hat G,\Psi_{w\theta}^+,\dotw\theta,E,q) &=
    \iota^* D(\hat G',\Psi_{w'\theta'}^+,\dotw'\theta',E',q^k).
    \end{align*}
    We may now assume that $\hat G$ is simple.

\subsubsection{isogeny reduction}

Next, we give a reduction that removes the assumption that $\hat G$ is
simply connected.  That is, it is enough to prove the proposition for
any group in the isogeny class of $\hat G$.  (This step is not
required in the proof of the proposition, but we include it as a
reduction that is quite useful when making explicit calculations of
the data $\D$ and its partition function.)  Suppose that we have a
surjective map $\hat G_{sc}\to \hat G$ with kernel $Z' \subseteq
Z(G_{sc})$.  Assume that we have data $\D=(\hat U,\ldots)$ for $\hat
G$.  We show how to construct data $\D_{sc}=(\hat U_{sc},\ldots)$. for
$\hat G_{sc}$, adding the subscript ${sc}$ to all data attached to
$\hat G_{sc}$.  The morphism $\hat U \to \hat S\times \hat T_1$ gives
$X_*(\hat U)\to X_*(\hat S)\times X_*(\hat T_1)$.  Define $\hat
U_{sc}$ by defining its cocharacter lattice to be the preimage in
$X_*(\hat U)$ of
\[
X_*(\hat S_{sc})\times X_*(\hat T_{1, sc}) 
\subseteq X_*(\hat S)\times X_*(\hat T_1).
\] 
By restriction, we have a map
\[
X_*(\hat U_{sc})\to X_*(\hat S_{sc}) \times X_*(\hat T_{1,sc}).  
\]
The components of this map give a morphism $\phi_{sc}:\hat
U_{sc}\to\hat S_{sc}$ and an isogeny $\iota_{sc}:\hat U_{sc} \to \hat
T_{1,sc}$.  Fix any lift $\dotw_{sc}$ of $\dotw$ to the simply
connected cover.  Define $\hat B_{1,sc}$ and $\hat B_{sc}(w\theta)$ by
the natural bijection of roots between $\hat G$ and $\hat G_{sc}$.

Let $\epsilon_{sc}\in \hat S_{sc}$ be any lift of $\epsilon\in \hat
S$.  For every $\uu\in \hat U_{sc}$ with image $\tau\in \hat
T_{1,sc}$, the elements $\epsilon_{sc}\phi_{sc}(\uu)$ and
$\tau\dotw_{sc}\theta$ have the same image in $\hat S/W^\theta$.  For
each $z\in Z'$, define
\[
\hat U_{sc,z} = \{\uu \in \hat U_{sc}\mid 
z \epsilon_{sc}(\uu)\phi_{sc}(\uu)\theta\text{ and }
\iota_{sc}(\uu)\dotw_{sc}\theta \text{ same image in } \hat S_{sc}/W^\theta\}.
\]
It follows from the fact that conjugacy (property 1) holds for $\hat
G$ that
\[
\hat U_{sc} = \cup_{z\in Z'} \hat U_{sc,z},
\]
expressing an irreducible set $\hat U_{sc}$ as a finite union of
Zariski closed subsets.  It follows that $\hat U_{sc} = \hat U_{sc,z}$
for some $z\in Z'$.  We replace $\epsilon_{sc}$ by $z\epsilon_{sc}$.
Then $\hat U_{sc} = \hat U_{sc,1}$.  That is, property (1-conjugacy) holds for
$\hat G_{sc}$.  Regularity (2) follows from property (2-regularity)
for $\hat G$.

Property (3-partition) follows because the determinant is computed
through the adjoint representation, and the groups $\hat G$ and $\hat
G_{sc}$ have the same adjoint group.  This completes the isogeny
reduction.

\subsubsection{completion of the proof}

We are ready to complete the proof.  By the reductions, we may assume
that $\hat G$ is simple and simply connected, and that no further Levi
descent is possible.

We consider two cases, depending on whether $\theta$ is trivial.

Assume first that $\theta$ is nontrivial.  Under the given
assumptions, $w\theta$ and $\theta$ are necessarily conjugate
automorphisms of the extended Dynkin diagram.  That is, $w_1 w \theta
w_1^{-1} = \theta$ for some Weyl group element $w_1$.  By Section
\ref{sec:conjred}, we can assume that $w=1$.  In this case, the
construction of data $\D$ satisfying the proposition is trivial: $\hat
U = \hat T_1 =\hat S$, $\hat B(w\theta) = \hat B_1 = \hat B$,
$\epsilon = \dotw = 1$, etc.

Now assume that $\theta=1$.  Under the assumption that no Levi descent
is possible, we find that $w$ acts transitively on the nodes of the
extended Dynkin diagram, $\hat G$ has Dynkin diagram $A_{n-1}$, and
$w$ is the Coxeter element of the Weyl group.  We have $\hat T = \hat
S$, $\hat T^w$ is finite, and $\hat T_1 = 1$.  Then $\hat U = 1$,
$\phi:1\to \hat S$, and $\iota:1\to 1$ are uniquely determined.  The
choice of $\dotw$ does not matter; all lifts $\dotw$ of $w$ are
conjugate.  Let $\epsilon$ be a conjugate of $\dotw$ in $\hat T$.  It
is easy to check that the Coxeter element $\dotw$ and $\epsilon$ are
regular.  Properties (1-conjugacy) and (2-regularity) then hold.

All roots have $N_1$-norm $0$.  Let $\hat B(w) = \hat B_1 = \hat B$,
the subgroup of upper triangular matrices.  Then $\Psi_{\phi,\theta}^+
= \Psi_{w\theta}^+=\emptyset$.  Both determinants are $1$ (on a
$0$-dimensional vector space).  So property (3-partition) holds.

This completes the proof of the proposition.
\end{proof}

\subsection{endoscopic partition function}

A Weyl group element that acts on the extended Dynkin diagram arises
in the following context.  Let $G$ be an unramified $p$-adic reductive
group, and let $H$ be an unramified endoscopic group of $G$.  We
assume that we are given an embedding
\[
\xi:{}^LH\to {}^LG,
\]
that factors through a finite unramified extension of $F$:
\[
\xi:\hat H\rtimes \ang{\theta_H}\to \hat G\rtimes \ang{\theta},
\]
such that $\xi(\theta_H) = \dotw \rtimes \theta,$ for some
representative $\dotw $ of an element $w$ in the Weyl group $W$.  It
is known that the element $w$ can be chosen to act as an automorphism
of the extended Dynkin diagram, up to an equivalence of endoscopic
data \cite[\S4.7]{hales1993simple}.  We use notation from Section
\ref{sec:theta-conj}: $\theta_1 = \dotw\theta$, $\hat S$, $\hat T_1$,
$N$, $N_1$, etc.

From the description of endoscopic data, we may assume that
$\hat H = C_{\hat G}(s)^0$ for some $s\in \hat T$, and that
$\xi(h) = h$, for $h \in \hat H$, with this identification.  We may
assume $\hat T = \hat T_G = \hat T_H$ using this identification of
$\hat H$ with a subgroup of $\hat G$.  In what follows all
$L$-morphisms $\xi$ are assumed to have this form.  Because
$\hat H = C_{\hat G}(s)^0$ for some $s\in \hat T$, the root system
$\Psi_H$ of $\hat H$ with respect to $\hat T$ is a subset of $\Psi$.

\begin{lemma}   If $\alpha\in \Psi$ and  $N_1\alpha=0$, then
$\alpha$ is not in the root system of $\hat H$.
\end{lemma}

\begin{proof} We prove the contrapositive.  Assume that
  $\alpha\in\Psi_H$.  Pick a Borel subgroup $\hat B_H\supseteq \hat T$
  of $\hat H$ that is $w\theta$-stable.  Replace $\alpha$ by $-\alpha$
  if necessary so that $\alpha\in\Psi^+(\hat T,\hat B_H)$.  Then
  $N_1\alpha$ is a sum of positive roots, hence positive.  Thus,
  $N_1\alpha$ is nonzero.
\end{proof}

If $\hat B(w\theta)$ is an adapted Borel subgroup of $\hat G$, then we
can define a positive root system $\Psi_H^+$ for $\hat H$ by
$\Psi^+(\hat T,\hat B(w\theta))\cap \Psi_H$.  We say that such a
system of positive roots for $\hat H$ is adapted.  By the two previous
lemmas, if $\Psi^+_H$ is adapted, then $w\theta$ preserves $\Psi^+_H$.

We can use this construction to define an {\it endoscopic} partition
function as follows.  We have a disjoint sum
\begin{equation}\label{eqn:disj-b1}
\Psi^+(\hat T,\hat B(w\theta)) = 
\Psi^+_H \sqcup \Psi_{w\theta,0}^+ 
\sqcup \Psi_{w\theta}^+(\hat G\setminus\hat H),
\end{equation}
where $\Psi_{w\theta,0}^+$ is the set of $\hat B(w\theta)$-positive
roots $\alpha$ such that $N_1\alpha=0$, and $\Psi_{w\theta}^+(\hat
G\setminus\hat H)$ is the set of positive roots with nonzero
$N_1$-norm that are not roots of $\hat H$.  We define the {\it
  endoscopic partition function} to be
\begin{equation}\label{eqn:endo-partition}
P(\hat G,\Psi_{w\theta}^+(\hat G\setminus\hat H),\dotw\theta,E,q).
\end{equation}
As we will see, the branching rule for the subgroup
$\xi({}^LH)\subseteq {}^LG$ is expressed in terms of this partition
function.

We have constructed an adapted set $\Psi^+_H$ of positive roots of
$\hat H$.  We expand the endoscopic partition function (or rather its
pullback to $\hat U$) in a series
\begin{equation}
\iota^* P(\hat G,\Psi_{w\theta}^+
(\hat G\setminus\hat H),\dotw\theta,E,1) = \sum_{\mu} p_\mu e^\mu,
\end{equation}
where the support of $\mu\mapsto p_\mu$ is a subset of $X^*(\hat U)$.

\subsection{branching rules}

The irreducible representations of ${}^LG$ restricted to $\hat
G\rtimes\theta$ are classified by a highest weight $\lambda\in
P^+_G\subseteq Y^*_G$ (with character $\tau_\lambda$), and similarly
for irreducible representations on $\hat H\rtimes \theta_H$, with
$\mu\in P^+_H\subseteq X^*(\hat T_1)$ (with character $\sigma_\mu$).
This section gives a branching rule for $\tau_\lambda$ restricted to
$\xi(\hat H\rtimes\theta_H)$, as a sum of $\sigma_\mu$:
\[
J_\xi\tau_\lambda = \sum_\mu m(\lambda,\mu) \sigma_\mu
\]
for some coefficients $m(\lambda,\mu)$ and $J_\xi$ defined as in
(\ref{eqn:Jxi}).

If $\sigma_\mu$ is an irreducible character and
$\chi:\ang{\theta_H}\to \ring{C}^\times$ is a multiplicative
character, then $ \chi\otimes\sigma_\mu$ is again an irreducible
character.  Restricted to the component $\hat H\rtimes\theta_H$, the
characters $\chi\otimes\sigma_\mu$ and $\sigma_\mu$ are linearly
dependent: $\chi\otimes\sigma_\mu = \chi(\theta_H) \sigma_\mu $.  This
means that the multiplicities $m(\lambda,\mu)$ should take values in
$\ring{Z}[\zeta]$, where $\zeta$ is a primitive root of unity of the
same order as $\theta_H$.

We fix data $\D=(\hat U,\hat B(w\theta),\hat
B_1,\iota,\phi,\epsilon,\dotw)$ associated with $\xi:{}^LH\to{}^LG$ as
in Proposition \ref{lemma:ephi}.  For the moment, we assume that
$\dotw$ associated with $\D$ coincides with $\dotw$ associated with
$\xi$.  We have a disjoint sum decomposition
\[
\Psi^+(\hat T,\hat B_1) = \Psi_{\phi,\theta}^+\sqcup \Psi_{\phi,\theta,0}^+,
\]
where 
\[
\Psi_{\phi,\theta,0}^+ = 
\{\alpha\in \Psi^+(\hat T,\hat B_1)\mid \phi^* N\alpha=0\}.
\]
The condition $\phi^* N\alpha=0$ implies that $\phi^*_\epsilon D(\hat
G,\Psi_{\phi,\theta,0}^+,\theta,E,1)$ is a constant
$d_0(\epsilon,\theta)\in\ring{Q}(\zeta)$ (that is, it is independent
of $\mu\in X^*(\hat U)$).  Regularity (Proposition \ref{lemma:ephi})
implies that the constant is nonzero.  Evaluation of the constant is
routine, but we do not do so here.  We abbreviate
\begin{align*}
D_H := \iota^* D(\hat H,\Psi^+_H,\dotw\theta,E,1) 
= \iota^* D(\hat G,\Psi^+_H,\dotw\theta,E,1).
\end{align*}
This is the denominator in the twisted Weyl character formula (on
$X^*(\hat U)$) for $(\hat H,\theta_H)$ with respect to the positive
root system $\Psi^+_H$.  Combining these identities and the
proposition, we have
\begin{align}\label{eqn:DHDG}
\begin{split}
D_H &= i^* D(\hat H,\Psi^+_H,\dotw\theta,E,1)\\
&=\frac{i^*D(\hat G,\Psi_{w\theta}^+,\dotw\theta,E,1)}
{i^*D(\hat G,\Psi_{w\theta}^+(\hat G\setminus\hat H),\dotw\theta,E,1)}\\
&=
{i^*P(\hat G,\Psi_{w\theta}^+(\hat G\setminus\hat H),\dotw\theta,E,1)}
\phi^*_\epsilon D(\hat G,\Psi_{\phi,\theta}^+,\theta,E,1)\\
&=
{i^*P(\hat G,\Psi_{w\theta}^+(\hat G\setminus\hat H),\dotw\theta,E,1)}
\frac{\phi^*_\epsilon D(\hat G,\Psi^+(\hat T,\hat B_1),\theta,E,1)}
{\phi^*_\epsilon D(\hat G,\Psi_{\phi,\theta,0}^+,\theta,E,1)}\\
&=
(\sum_{\mu} p_\mu e^\mu) 
\ \frac{\phi^*_\epsilon D_G}{
d_0(\epsilon,\theta)},
\end{split}
\end{align}
with abbreviation $D_G= D(\hat G,\Psi^+(\hat T,\hat B_1),\theta,E,1)$
for the twisted Weyl denominator for $(\hat G,\theta)$ with respect to
the positive root system $\Psi^+(\hat T,\hat B_1)$.  The isogeny
$\iota:\hat U\to\hat T_1$ gives $X^*(\hat T_1)\subseteq X^*(\hat U)$.

Define constants $m(\lambda,\mu)$, for $\lambda\in P^+\subseteq
X^*(\hat S)$ and $\mu\in P_H^+ \subseteq X^*(\hat T_1)\subseteq
X^*(\hat U)$:
\begin{equation}\label{eqn:branch}
m(\lambda,\mu) = \sum_{w\in W^\theta} 
(-1)^{\ell w} ({w\bullet\lambda})(\epsilon)
\frac{p_{\mu-\phi^*(w\bullet\lambda)}}{d_0(\epsilon,\theta)} \in \ring{Q}(\zeta).
\end{equation}

The following theorem is the main result of this section.

\begin{theorem}\label{thm:branch}
  Let $G$ be a reductive group and let $H$ be an endoscopic group of
  $G$, both unramified.  Let $\xi:{}^LH\to {}^LG$ be an embedding of
  $L$-groups that factors over a finite unramified extension $E/F$.
  Suppose that $\xi(\theta_H) = \dotw\theta$.  Let $\dotw\mapsto w\in W$.
  Let $\D=(\hat U,\ldots,\dotw)$ be the data constructed by
  Proposition \ref{lemma:ephi} (with the same $\dotw$ for $\xi$ and $\D$).  
  Then Equation \ref{eqn:branch} gives
  the twisted branching rule for $\hat G\rtimes \theta$ restricted to
  $\xi(\hat H\rtimes \theta_H)$:
\[
J_\xi \tau_\lambda = \sum_\mu m(\lambda,\mu) \sigma_\mu \in
\ring{C}[X^*(\hat U)].
\]
Both sides are supported on $X^*(\hat T_1)\subseteq X^*(\hat U)$.
\end{theorem}

Before starting the proof of the theorem, we give a transformation
rule describing how the branching rules depend on the choice
$\dotw\mapsto w$ used to define the embedding $\xi$ of $L$-groups.

\begin{lemma}\label{lemma:transform}
  Suppose that we have two embeddings $\xi_i:{}^LH\to {}^LG$ with
  $\xi_1(\hat H) =\xi_2(\hat H)$, $\xi_1(\hat T) = \xi_2(\hat T) =
  \hat T$, $\xi_i(\theta_H) = \dotw_i\theta$, and $\dotw_2 =
  t\dotw_1$.  Write $m_i(\lambda,\mu)$ for the twisted branching
  coefficients for $\xi_i$.  Then $m_1(\lambda,\mu) =
  m_2(\lambda,\mu)\mu(t)$.
\end{lemma}

We note that the image $\xi_i({}^LH)$ does not not depend on $i$, but
the branching formula does.

\begin{proof}[Proof (lemma)]  
  Let $V_\lambda$ be the representation of $G\rtimes \ang{\theta}$
  with highest weight $\lambda$, normalized as usual by the condition
  that $\theta v=v$, when $v$ is a vector of highest weight in
  $V_\lambda$.

  Let $W$ be a $\hat H$-irreducible subrepresentation of $V_\lambda$
  with highest weight $\mu$.  By the condition $\xi_1(\hat H)
  =\xi_2(\hat H)$, the two module structures on $W$ agree.  Assume
  that $\mu$ is a $\theta_H$-fixed weight.  We extend $W$ to ${}^LH$
  with the usual normalization $\theta_H v = v$, where $v$ is a vector
  of highest weight.  The normalization depends on $i$.  The two
  different normalizations
\[
\theta_H v = \xi_i(\theta_H) v = \dotw_i \theta v =v,
\]
for $i=1,2$, 
differ by a scalar $\mu(t)$.  
Hence the multiplicities also transform by a factor $\mu(t)$.
\end{proof}

In general, the $L$-group can be formed with respect to various
unramified Galois extensions.
We let $L/F$ be a second unramified extension with $L/E/F$.  Let
$[L:E]=\ell$.  We show that the  branching multiplicities on
do not depend on $L$.

Fix an admissible embedding
\[
\xi_L:\hat H \rtimes \op{Gal}(L/F)\to \hat G \rtimes \op{Gal}(L/F),
\]
where $\xi_L(\Frob_L) = \dotw \Frob_L$, with the same choice $\dotw $
as with $\xi = \xi_E$.  Then $\Frob_L^\ell$ acts trivially on the
datum of $\hat H$, and $(\dotw \Frob_L)^\ell$ acts trivially on the
datum of $\hat G$.  Let $\tau_\lambda$ be an irreducible
representation of $\hat G$ that is $\theta$-fixed.  The extension of
$\tau_\lambda$ to $\hat G\rtimes \op{Gal}(L/F)$ factors through $\hat
G\rtimes \op{Gal}(E/F)$.  Similarly, $\sigma_\mu$ extends to $\hat
H\rtimes \op{Gal}(L/F)$ and factors through $\hat
H\rtimes\op{Gal}(E/F)$.  We conclude that the branching multiplicities
are the same for $L/F$ and
\[
\xi_E(\hat H \rtimes \Frob_E) 
= \xi_L(\hat H\rtimes \Frob_L)/(\dotw \Frob_L)^\ell 
\subset (\hat G\rtimes \Frob_L)/( \Frob_L)^\ell 
= \hat G\rtimes \Frob_E.
\]
The multiplicity formula is independent of the choice of $L/F$.

\begin{proof}[Proof (theorem)]
  We extend Kostant's formula for branching multiplicities
  $m(\lambda,\mu)$ to this setting, following Goodman and Wallach
  \cite[\S8.2.2]{goodman}.  Recall that this formula for
  $m(\lambda,\mu)$ is based on the (twisted) Weyl character formula.

  The restriction $J_\xi\tau_\lambda$ is determined on the coset $\hat
  T\dotw \theta$, which depends only on $\hat T_1 \dotw\theta$ or its
  pullback to $\uu\mapsto \iota(u)\dotw\theta$ to $\uu\in \hat U$.  By
  the proposition, these elements are conjugate to $\uu\mapsto
  \epsilon\phi(\uu)\theta$.  So
\[
J_\xi\tau_\lambda (\iota(\uu)\dotw \theta) = (\phi_\epsilon^*\tau_\lambda)(\uu\theta).
\]
By convention, we drop $\theta$ (resp. $\theta_H$) from the notation
in twisted identities on the $\theta$-component of the group, writing
the equation simply as $\iota^*J_\xi\tau_\lambda =
\phi_\epsilon^*(\tau_\lambda)$ on $\hat U$.

We show that the properties of $\phi,\epsilon$ imply the branching
rule, which we compute using Equation \ref{eqn:DHDG} and the twisted
Weyl character formula on $\hat H$ and $\hat G$.
\begin{align*}
(\phi^*_\epsilon\tau_\lambda) D_H 
 &= (\iota^*J_\xi\tau_\lambda) D_H\\
 &= \sum_{\mu'} m(\lambda,{\mu'}) J_H(e^{\mu'})\\
  &= m(\lambda,\mu)e^\mu + \sum_{\mu'\ne\mu} c_{\mu'} e^{\mu'}.\\
(\phi^*_\epsilon\tau_\lambda) D_H 
  &= (\sum p_{\mu'} e^{\mu'})(\phi^*_\epsilon(\tau_\lambda D_{G})) /d_0\\
  &= (\sum p_{\mu'} e^{\mu'}) \phi^*_\epsilon(J_G(e^\lambda)) /d_0\\
  &= \sum_{\mu'} \sum_{w\in W^\theta} (-1)^{\ell w} 
  (p_{\mu'} e^{\mu'}) (\phi^*_\epsilon (e^{w\bullet \lambda})) /d_0\\
  &= \sum_{\mu'} \sum_{w\in W^\theta} (-1)^{\ell w}  
 ({w\bullet\lambda})(\epsilon) p_{\mu'}
e^{\mu'+\phi^*(w\bullet \lambda)} /d_0\\
  &= \sum_{\mu'} \sum_{w\in W^\theta} 
(-1)^{\ell w} ({w\bullet\lambda})(\epsilon) 
 p_{\mu' - \phi^*({w\bullet \lambda})} e^{\mu'} /d_0.
\end{align*}
We have used the twisted Weyl character formula with respect to
$\Psi^+_H$ on $\hat H$ and with respect to $\Psi^+(\hat T,\hat B_1)$
on $\hat G$.  To justify the equation in the third row, let $\mu\in
P_H^+$.  If $w'\bullet \mu' = w\bullet \mu$ for some $w,w'\in W_H$ and
$\mu'\in P_H^+$, then using the fact that $P_H^+$ is a fundamental
domain for $W_H$ and that $\mu'+\rho_H$ lies in the interior of that
domain, we find that $w=w'$ and $\mu=\mu'$.

Equating coefficients of $e^\mu$, we get Equation \ref{eqn:branch}.
\end{proof}

\section{Motivic Integration}

This section reviews the theory of motivic integration as developed by
Cluckers and Loeser~\cite{cluckers2008constructible}.

\subsection{the Denef-Pas language}

The Denef-Pas language is a three-sorted first-order formal language
in the sense of model theory.  Its intended structures are triples
$(F,k_F,\ring{Z})$, where $F$ is a valued field with discrete
valuation, $k_F$ is the residue field of $F$, and the value group of
$F$ is the ring of integers $\ring{Z}$.  The three sorts are $VF$
(the valued-field sort), $RF$ (the residue-field sort), and $\ring{Z}$
(the value-group sort).

In general, a first-order formal language is specified by sets of
relation symbols and function symbols.  The Denef-Pas language has the
following relation and function symbols.  The valued-field sort $VF$
has the symbols of the first-order language of rings $(0,1,+,\times)$.
The residue field sort also has the symbols of the first-order
language of rings.  The value-group sort is the Presburger language of
an ordered additive group with symbols $(0,+,\le,\equiv_n)$.  Here
$(\equiv_n)$ is a binary relation symbol for each $n\ge 2$, which is
to be interpreted as congruence modulo $n$ in $\ring{Z}$.  In
addition, there are two function symbols $\op{ord}:VF\to\ring{Z}$
(interpreted as the valuation on the valued-field) and $\op{ac}:VF\to
RF$ (interpreted as the angular component map).  For the structure
$(K((t)),K,\ring{Z})$, where $K((t))$ is the field of formal Laurent
series, the intended interpretation of $\op{ac}$ is the function
$\sum_{i\ge N} a_i t^i\mapsto a_N$ that returns the first nonzero
coefficient of the Laurent series (and sending $0\in K((t))$ to $0$).

First-order languages are constructed in the usual way, with formulas
built from logical connectives $(\land)$, $(\to)$, $(\lor)$, $\neg$,
equality $(=)$, variables of the three sorts, function symbols,
relation symbols, existential quantifiers of each sort, and universal
quantifiers of each sort.

Following the terminology of \cite{gordon}, we call a {\it fixed
  choice} any set-theoretic data that does not depend in any way on
the Denef-Pas language, its variables, nor on the structures of $VF$
and $RF$.  Examples of fixed-choices that appear in this paper are
Weyl groups, abstract groups, representations of split reductive
groups over $\ring{Q}$, and root systems.

\subsection{motivic integration}

Let $\op{Field}_\Q$ be the category of fields of characteristic zero.

Cluckers and Loeser have used the Denef-Pas language to define various
categories.  In particular, there is a category $\op{Def}_\Q$ of {\it
  definable subassignments}, given as follows.  Let
$\ring{N}=\{0,1,2,\ldots\}$.  For each $(m,n,r)\in\ring{N}^3$, let
$h[m,n,r]$ be the functor from $\op{Field}_\Q$ to the category of sets
that assigns to each field $K$, the set $h[m,n,r](K)=K((t))^m\times
K^n\times \ring{Z}^r$.  A {\it subassignment} of this functor is by
definition, a subset $S(K) \subseteq h[m,n,r](K)$ for each
$K\in\op{Field}_\Q$.  A definable subassignment $S$ is a subassignment
for which there exists a formula $\phi$ in the Denef-Pas language such
that for each $K\in\op{Field}_\Q$, the set of solutions of $\phi$ in
$h[m,n,r](K)$ is $S(K)$.  The definable subassignments are the objects
of the category $\op{Def}_\Q$.  A morphism $\phi:X\to Y$ is a
definable subassignment
\[
\phi\subseteq X\times Y
\subseteq h[m,n,r]\times h[m',n',r'] = h[m+m',n+n',r+r']
\]
that is the graph of a function $X(K)\to Y(K)$ for each $K\in
\op{Field}_\Q$.
A {\it free parameter} refers to a collection of free variables of the
same sort in a formula in the Denef-Pas language, ranging over a
definable subassignment.  

For each definable subassignment $X\in \op{Def}_\Q$, Cluckers and
Loeser have defined a ring $C(X)$ of {\it constructible motivic
  functions}.  The construction of this ring is a major undertaking,
and we refer the reader to their articles for details
\cite{cluckers2008constructible}.  The elements of this ring are
called constructible motivic functions.  Although they behave in many
ways as functions on $X$, the elements of the ring are not literal
functions in the set-theoretic sense of function.

If $\phi:X\to Y$ is a morphism of definable subassignments, there is a
pullback of functions $\phi^*:C(Y)\to C(X)$.  The pullback $\phi^*$ is
a ring homomorphism, and pullbacks compose: $(\phi\psi)^* = \psi^*
\phi^*$.

If $X\to S$ is a morphism of definable subassignments, there is a
subgroup $I_S C(X)$ of $S$-integrable constructible motivic functions.
The intuitive interpretation of an $S$-integrable function $f$ is a
function such that the integral over each fiber of $X\to S$ is
convergent with respect to the canonical motivic measure.  For a
morphism $\phi: X\to Y$ over $S$, there is a pushforward
$\phi_!:I_SC(X)\to I_SC(Y)$ that is called {\it integration over
  fibers}.  Pushforwards compose: $(\phi\psi)_! = \phi_!\psi_!$.  In
this article, we always deal with bounded constructible functions.
Such functions are always integrable
\cite[Prop~12.2.2]{cluckers2008constructible}.  Thus, we do not need
to deal with integrability issues.

\subsection{Presburger constructible functions}

The ring $C(X)$ of constructible functions is the graded algebra
associated with a filtration on a tensor product $P(X) \otimes Q(X)$.
In terms of the three sorts of the Denef-Pas language, data related to
the value-group sort $\ring{Z}$ is encoded in $P(X)$ and data related
to the residue field sort $RF$ is encoded in $Q(X)$.  The left-hand
side $P(X)$ is a ring of {\it Presburger constructible functions.}
Every Presburger constructible function $f$ gives a constructible
motivic function $f\otimes 1$.

Much of what we do in this article is related to constructible
functions on integer lattices.  For this, we work with Presburger
constructible functions rather than the entire ring of constructible
motivic functions.

\subsection{volume forms}

Cluckers and Loeser have an extension of motivic integration that
allows integration with respect to volume forms
\cite[\S8]{cluckers2008constructible}.  In brief, there is a notion of
differential forms on a definable subassignment and a space of
definable positive volume forms.  Each differential form $\omega$ of
top degree has an associated volume form $|\omega|$.  For each
morphism $\phi:X\to Y$ over $S$, the pushforward $\phi_!$ extends to a
pushforward $f \mapsto \phi_!(f,\omega)$ with respect to the volume
form.  It is to be interpreted loosely as integration over the fibers
of $\phi:X\to Y$ with respect to a volume form constructed from a
Leray residue of $\omega$ on the fiber.

\subsection{$p$-adic specialization}

Let $\C$ denote the class of $p$-adic fields.  Let $\C_N\subseteq \C$
denote the subclass of fields whose residue characteristic is at least
$p\ge N$.

In general, we only care about what occurs in fields in $\C_N$ for $N$
arbitrarily large.  To make this precise, suppose that we have for
some $N$, a function $X$ with domain $\C_N$.  Then by restriction of
domain $\C_i$ to $\C_{j}$, for $N\le i\le j$, we may take the filtered
colimit of $X_i=X|_{\C_i}$.  Two functions $X$, $X'$ have the same
filtered colimit if they are equal in $\C_i$ for some sufficiently
large $i$.

Let $X$ be a definable subassignment of $h[m,n,r]$, and let $f$ be a
constructible motivic function on $X$.  There exists an $N$ such that
for all $F\in \C_N$, there are specializations
\[
X(F)\subseteq F^m\times k_F^n\times \ring{Z}^r,  
\quad f_F: X(F) \to\ring{C}.
\]
Only the  filtered colimits of $X$ and $f$ matter to us.

We warn the reader of a notational overload; we write $X(K)$ or $X(F)$
as $K$ and $F$ range over two quite different classes of fields.
Different symbols $K$ and $F$ disambiguate the context.  When $K\in
\op{Field}_\Q$, the valued field is $K((t))$ and the residue field is
$K$; but when $F$ is a $p$-adic field, $F$ is the valued field and its
residue field is denoted $k_F$.  We also warn that $K$ is used both
for a hyperspecial subgroup and for $K\in\op{Field}_\Q$.

The specializations have various expected properties.  If $\phi:X\to
Y$ is a morphism of definable subassignments over $S$, then we have
functions $\phi_F:X(F)\to Y(F)$.  When $f$ is $S$-integrable on $X$,
integration $\phi_!(f)$ over fibers specializes to integration over
fibers with respect to a canonical measure in $p$-adic fields $F\in
\C_N$ (for some $N$ depending on $\phi$).

The functions $f_F:X(F)\to\ring{C}$ that come from constructible
motivic functions $f\in C(X)$ have a special form
\begin{equation}\label{eqn:q}
f_F(x) = \sum_i \card(Y_i(F,x)) q_F^{\alpha_{i,F}(x)} 
\prod_j \beta_{i,j,F}(x)\prod_k \frac{1}{1-q_F^{a_{i,k}}},
\end{equation}
where all sums and products are finite, $\alpha_{i}:X\to\ring{Z}$,
$\beta_{ij}:X\to\ring{Z}$ are definable, $q_F$ is the cardinality of
the residue field of $F$, and $a_{i,k}$ are nonzero integers
\cite[\S2]{cluckers2011btransfer}.  The filtered colimits of these
functions are {\it $q$-constructible functions}.  Let $C_q(X)$ be the
space of $q$-constructible functions on $X$.  Sometimes we call the
specialization of a definable subassignment a definable set.  There is
an element $\ring{L}$, called the Lefschetz motive, in the ring of
constructible motivic functions that specializes to $q_F$ for every
$p$-adic field $F$.  When the first factors $Y_i$ are absent from
Equation \ref{eqn:q}, the function $f$ is Presburger constructible.

We warn the reader that very different constructible motivic functions
can yield the same $q$-constructible function.  For example, let
$[S]\in C(\op{pt})$ be the isomorphism class in the residue sort of
the set of nonzero squares, considered as a constructible motivic
function on a point.  Similarly, let $[N]$ be the class of the set of
nonsquares.  Then, under specialization to $p$-adic fields, the two
functions are equal: $[S](F) = [N](F) = (q_F-1)/2$, for
$F\in\C_1$. However, $[S]$ and $[N]$ are not at all the same
constructible motivic function. Indeed, their values on algebraically
closed residue fields $K$ are not equal: $[N](K)$ is the empty set and
$[S](K) = K^\times$ is not.  Another family of examples is provided by
isogenous elliptic curves.  They have the same number of points in a
finite field, but they are not generally isomorphic curves.  If a
constructible motivic function specializes to a $q$-constructible
function that is identically zero, then we call it a {\it null function}.

The theory of motivic integration specializes to $q$-constructible
functions. To integrate a $q$-constructible function, we lift it
to a constructible motivic function, use Cluckers-Loeser integration
there, then take its specialization again.  Two different lifts differ
by a null function, and its integral is also a null function.
Thus, integration of $q$-constructible functions is well-defined.

\subsection{definable reductive groups}\label{sec:defred}

Definable reductive groups are understood in the sense of
\cite{cluckers2011transfer}, \cite{gordon}.  In this work we restrict
to unramified reductive groups (quasi-split and split over an
unramified extension).

In the definable context, a reductive group $G\to Z$ lies over a
definable subassignment $Z$ called the {\it cocycle space} of $G$.  In
the case of an unramified reductive group that splits over an
unramified extension of degree $r$, we can take $Z\subseteq h[m,0,0]$,
for some $m$.  The set $Z$ parameterizes lists of coefficients of
irreducible monic polynomials, each defining a degree $r$ unramified
extension of $F$.  A field extension $E/VF$ of degree $r$ is
identified with $VF^r = VF[x]/(p)$, as $p$ runs over irreducible
polynomials parameterized by $Z$.

Recall that there is no Frobenius map in the context of the Denef-Pas
language, because it is not possible to take a $q$th power.  Instead,
we choose a generator of the Galois group of an unramified extension
$E/VF$ and call it the quasi-Frobenius element.  As part of the
cocycle space data $Z$, we assume we are given a quasi-Frobenius
element $\op{qFrob}$ that corresponds to the automorphism $\theta$ of
$\hat G$.

A connected split reductive group is treated as a definable
subassignment through a faithful representation of the group.  The
group is identified with a closed subgroup of $\op{GL}(n,F)$.
Quasi-split reductive groups that split over an unramified degree $r$
extension (parameterized by a cocycle space $Z$) are defined in terms
of explicit representations of those groups in $\op{GL}(n,E)$, where
$E/VF$ is treated as above.

If $G$ is an unramified reductive group, we may construct a
hyperspecial subgroup $K$ as a definable subassignment of
$G$~\cite{cluckers2011local}.

A quasi-split reductive group $G$ carries an invariant differential
form $\omega$ of top degree, which is described in the context of
definable subassignments in \cite{gordon}.  All integration in this
article is assumed to be carried out with respect to invariant
measures.  We have the invariant integral $\vartheta_!(f,\omega)$ of a
constructible integrable function $f\in C_q(G)$ with respect to the
morphism $\vartheta:G\to\{\op{pt}\}$ to a point using the invariant
differential form $\omega$.

\subsection{enumerated Galois groups}\label{sec:galois}

We deal with field extensions and Galois groups in the way described
in \cite{gordon} and \cite{cluckers2011transfer}.  We let $\Gamma$ be
an abstract group with fixed enumeration $1=\sigma_1,\ldots,\sigma_n$
of its elements.  We assume a fixed short exact sequence
\[
1\to \Gamma^t\to\Gamma\to\Gamma^{unr}\to 1,
\]
with $\Gamma^t$ and $\Gamma^{unr}$ both cyclic.  The group $\Gamma$
plays the role of a Galois group with inertia subgroup $\Gamma^t$ and
unramified quotient $\Gamma^{unr}$.  We treat this data as an abstract
fixed choice, without a priori connection to the Galois group of any
particular extension of $p$-adic fields.

We may fix an abstract root datum and choose an action of $\Gamma$ on
the root datum, stabilizing the set of simple roots.  Through this
action on the root datum, $\Gamma$ acts on the Weyl group, and we may
construct the semidirect product $W\rtimes \Gamma$.

By abstract unramified Galois group we mean a fixed finite cyclic
group $\Gamma=\Gamma^u$ with choice of generator $\op{qFrob}$ that we
call the {\it quasi-Frobenius} element.  By abuse of terminology, we use
the word quasi-Frobenius to refer either to the generator of $\Gamma^u$
or as its realization as a matrix with values in the valued field sort $VF$,
as described in \cite{cluckers2011transfer}.

In this article, the abstract dual group is the Langlands dual
constructed with respect to $\Gamma$ and $\op{qFrob}$ rather than the
Galois (or Weil) group of a field.

\subsection{definability results}\label{sec:definability}

In this section we assume that $G$ is an unramified connected
reductive group.  It is treated as definable subassignment over a
definable cocycle space $Z$.

Standard subgroups of $G$ such as a hyperspecial $K$, Borel subgroup
$B$, $T\subseteq B$, the unipotent radical $N$ of $B$, the maximal
split subtorus $A$ of $T$ are all definable.  In the following lemmas
a field extension $L/VF$ often appears.  We can treat it in a
definable way as in Section \ref{sec:defred} whenever we have an a
priori bound on the degree $L/VF$.  In each case that follows we have
an a priori bound on the degree of the extension.  Similar remarks
apply to unramified field extensions $E/VF$ that appear.

\begin{lemma}  
  Let $G$ be an unramified reductive group.  There exists a definable
  subassignment of $G\times G$ of all pairs $(\gamma,x)$ such that
  $\gamma$ is semisimple (possibly singular) and $x$ lies in the
  identity component of the centralizer of $\gamma$.
\end{lemma}

\begin{proof}  The condition that $\gamma$ is semisimple 
is definable by the condition that $\gamma$ is conjugate to $T$ by
$G(L)$ for some field extension $L/VF$. 

The commutativity of $x$ and $\gamma$ is obviously definable.  We need
more to describe the identity component in definable terms.  We use
Steinberg's result that the centralizer of a semisimple element in a
simply-connected group is connected.  We work again over a field
extension $L/VF$, and write factorizations $x = z_x \bar x_{sc}$ and
$\gamma = z_\gamma \bar \gamma_{sc}$, where $z_x$ and $z_\gamma$ are
in the center of $G(L)$ and $x_{sc},\gamma_{sc}$ are in the simply
connected cover of the derived group $G_{sc}(L)$.  Then $x$ is in the
identity component of the centralizer iff $x_{sc}$ and $\gamma_{sc}$
commute.  This is a definable condition.
\end{proof}

\begin{lemma} 
  Let $G$ be an unramified reductive group.  There exists a definable
  subassignment of $G\times G$ of all pairs $(\gamma,\gamma')$ such
  that $\gamma$ is semisimple (possibly singular) and $\gamma'$ is
  stably conjugate to $\gamma$.
\end{lemma}

\begin{proof} The elements $\gamma$ and $\gamma'$ are stably conjugate
  iff there exists $g\in G(L)$ for some Galois extension $L/VF$ such
  that $\gamma^g = \gamma'$ and $\sigma(g)g^{-1} \in C_G(\gamma)^0$,
  where $\sigma\in \op{Gal}(L/VF)$.  The Galois group and its cocycles
  are treated within the Denef-Pas language as in Section
  \ref{sec:galois}.  The identity component is handled by the previous
  lemma.
\end{proof}

\begin{lemma} 
  Let $G$ be an unramified reductive group, given as a definable
  subassignment over a cocycle space $Z$.  There is a definable
  subassignment $G_{qs}$ (resp. $G_u$) of $G$ over $Z$ consisting of 
  semisimple elements $\gamma$ such that the identity component of
  the centralizer of $\gamma$ is quasi-split (resp. unramified).
\end{lemma}

\begin{proof} 
  A reductive group is quasi-split iff the Levi subgroup of the
  minimal parabolic subgroup is a Cartan subgroup.  This occurs iff
  the centralizer of the split component of the center of the Levi is
  a Cartan subgroup.

  The maximal split torus $A$ of $T$ is a definable set.  The
  condition for $a$ to belong to a split torus is definable by the
  condition that $a$ is conjugate to an element of $A$.

  The centralizer $C=C_G(\gamma)$ is quasi-split exactly when there
  exists $a\in C$ that is conjugate to an element of $A$ and such that
  its centralizer in $C$ is a Cartan subgroup of $G$.  This is a
  definable condition.

  Now turning to $G_u$, the group $C=C_G(\gamma)^0$ is unramified
  exactly when $C$ is quasi-split and there exists $t\in C$ whose
  centralizer is an unramified Cartan subgroup of $G$.  An unramified
  Cartan subgroup is one that is conjugate to $T$ by $G(E)$ for some
  unramified extension $E/VF$.  This is a definable condition in the
  Denef-Pas language.
\end{proof}

A {\it strongly compact element} $\gamma\in G(F)$ is an element that
belongs to a bounded subgroup of $G$.

\begin{lemma} 
The set of strongly
  compact semisimple elements is a definable subassignment.
  The set of topologically unipotent semisimple elements in a
  reductive group is a definable subassignment.  
\end{lemma}

\begin{proof} We can define a strongly compact semisimple element as
  one that is conjugate in $G(L)$ to an element $t\in T(L)$ of the
  split torus, for suitable large extension $L/VF$
  (of fixed degree), and such that the valuation of $\lambda(t)$ is
  $0$ for every $\lambda\in X^*(T)$.  It is enough to let $\lambda$ run
  over a finite set of generators of $X^*(T)$.  We can define topologically
  unipotent elements as elements conjugate to an element $t\in T$ by
  $G(L)$, such the valuation of $\lambda(t)$ is $0$ and the angular
  component of $\lambda(t)$ is $1$ for each $\lambda$.
\end{proof}

\begin{lemma} 
  Let $G$ be an unramified reductive group with unramified endoscopic
  group $H$, given as definable subassignments over a common cocycle
  space $Z$.  There is a definable subassignment of all pairs
  $(\gamma,\gamma_H)$ such that $\gamma_H$ is strongly $G$-regular in
  $H$ and $\gamma\in G_u$ is an image of $\gamma_H$.

  Moreover, consider the
  Denef-Pas statement $\psi$ that asserts that for all strongly
  $G$-regular elements $\gamma_H$ in $H$, there exists an image
  $\gamma\in G_{qs}$ that is an image of $\gamma_H$.  Then there
  exists $N$ such that $\psi_F$ is true for all $F\in\C_N$.
\end{lemma}

\begin{proof}  It is a definable condition to say that $\gamma$ is an
  image of $\gamma_H$.

  Moreover, by Kottwitz \cite[3.3]{kottwitz1982rational}, in
  characteristic zero, there exists $\gamma'$ in the stable conjugacy
  class of $\gamma$ such that $C_G(\gamma')$ is quasi-split.  The same
  result holds in large positive characteristic by the transfer
  principle (Section \ref{sec:transfer}).
\end{proof}

\begin{definition} 
  We say that $\gamma\in G(F)$ is {\it strongly semisimple} if $\gamma$
  is an element of a torus that splits over an unramified extension
  $E/VF$, if it is strongly compact, and if for every root $\alpha\in
  \Phi$ of $(T,G)$ and for every $g\in GL(E)$ such $\gamma^g\in T(E)$,
  we have
  \[
  \alpha(\gamma^g) = 1,\quad\text{or}
\quad \op{ord}(\alpha(\gamma^g)-1)=0.
  \]
 \end{definition}

 \begin{lemma} The set of strongly semisimple elements of an
   unramified reductive group $G$ is a definable set.
   \end{lemma}

   \begin{proof} This is clear from the definition and from the
     previous constructions.
   \end{proof}

   \begin{remark} The absolutely semisimple part $\gamma_s$ in the
     topological Jordan decomposition of an element $\gamma = \gamma_s
     \gamma_u$ is defined as a $p$-adic limit that cannot be treated
     within the Denef-Pas language \cite{hales1993simple}.  Instead,
     we allow $\gamma_s$ to run over a definable set of strongly
     semisimple elements.  We can no longer assert the uniqueness of
     the topological Jordan decomposition, but we obtain the existence
     of a definable decomposition, which is sufficient for our
     purposes.
    \end{remark}

\begin{lemma}[definable topological Jordan decomposition] 
  Let $G$ be an unramified reductive group.  There is a definable
  subassignment of triples $(\gamma,\gamma_s,\gamma_u)\in G^3$ such
  that $\gamma$ is strongly regular semisimple and strongly compact,
  $\gamma_s$ is strongly semisimple,
  $\gamma_u$ is
  topologically unipotent, 
  $\gamma = \gamma_s \gamma_u = \gamma_u\gamma_s$,  and
\[
\alpha(\gamma_s)=1,
\quad\text{ or }
\quad \op{ord}(\alpha(\gamma_s)-1)=0,
\]
for all  roots $\alpha$ of the Cartan subgroup $C_G(\gamma)$.

Moreover, consider statement $\psi$ in the Denef-Pas language that
asserts that for every strongly-regular strongly-compact semisimple
element $\gamma$, there exists $\gamma_s$ and $\gamma_u$ such that
$(\gamma,\gamma_s,\gamma_u)$ belongs to this definable set of triples.
Then there exists $N$ such that $\psi_F$ holds for all $F\in \C_N$.
\end{lemma}

\begin{proof} The definability follows by previous lemmas on the
  definability of the set of strongly compact elements,
  strongly-semisimple elements, and topologically unipotent elements.

The statement $\psi_F$ holds for $p$-adic fields of characteristic
zero by the existence of a topological Jordan decomposition
\cite{hales1993simple}.  It also hold in sufficiently large positive
characteristic by a transfer principle for statements in the Denef-Pas
language (Section \ref{sec:transfer}).
\end{proof}

\subsection{spherical Hecke algebra 
for an unramified definable group}

Let $G$ be a definable unramified reductive group over a cocycle space
$Z$.  Let $A$ be a maximal split torus in $G$ of dimension $r$.  We
identify its cocharacter lattice  $X_*(A)$ with $\ring{Z}^r$ by a
choice of free generators of $X_*(A)$.  This allows us to treat
$X_*(A)$ as the definable subassignment $h[0,0,r] = \ring{Z}^r$.  Let
$X^*(A)$ be the character lattice  of $A$.

There is a perfect pairing $\langle\cdot,\cdot\rangle:X^*(A)\times
X_*(A) \to \ring{Z}$.  For each $\lambda\in X_*(A)$, there is a
definable subassignment $A_\lambda \subseteq A$ given by the formula
\begin{equation}\label{eqn:alambda}
A_\lambda = \{ a \in A \mid \op{ord}(\mu(a)) 
= \ang{\mu,\lambda},\text{ for all } \mu\in X^*(A) \}.
\end{equation}
There is a definable subassignment of $X_*(A)\times A$ given by pairs
$(\lambda,a)$ such that $a\in A_\lambda$.  Of course, $p$-adically,
$A_\lambda$ is just the coset $\varpi^\lambda A(O_F)$, where $O_F$ is
the ring of integers of $F$.

Recall $P^+\subseteq X_*(A)$ is the set of cocharacters in the positive
Weyl chamber.

\begin{lemma} 
  $P^+$ is a definable subset (of $\ring{Z}^r$).
\end{lemma}

\begin{proof} 
  $P^+$ is defined by linear inequalities, which can be expressed in
  the Presburger language.
\end{proof}

\begin{lemma}[Cartan decomposition] \label{lemma:cartan}
  There is a definable subassignment of $P^+\times G$ given by 
\[
L_G = \{(\lambda,g)\in P^+\times G \mid g \in K A_\lambda K \}.
\]
The fiber $L_G(\lambda)$ over each $\lambda\in P^+$ is definable.
Moreover, $L_G(\lambda)\cap L_G(\lambda') = \emptyset$, for
$\lambda\ne \lambda'$.
\end{lemma}

By Bruhat-Tits, the Cartan decomposition holds over general discrete valued
fields \cite[4.4.3]{bruhat1972groupes}.

\begin{remark}   
  $L_G$ captures the entire spherical Hecke algebra as a single
  definable subassignment.  In applications to the fundamental lemma,
  it is important to work with this single subassignment rather than
  an infinite basis of the spherical Hecke algebra.
\end{remark}

We define the {\it spherical Hecke function} to be the characteristic
function of $L_G$.  It is a definable function (as well as a
$q$-constructible function) on $P^+\times G$.

The Satake transform $f\mapsto \hat f$ is an
isomorphism $\H(G//K)\to\ring{C}[Y^*]^{W^\theta}$.  Let
$s_{\lambda,\mu}$ be the coefficients of the change of basis $\hat
f_\lambda = \sum_\mu s_{\lambda,\mu} m_\mu$,
where $m_\mu$ is as before (\ref{eqn:mmu}).

The Satake transform lifts to the $q$-constructible setting.  The
Satake transform involves a term $q^{\langle\rho^\vee,\mu\rangle}$,
where $\rho^\vee = \rho(\Psi^\vee)$.  Constructible functions in the
formula (\ref{eqn:q}) only contain integral powers of $q$.  However,
\cite[\S B.3.1]{cluckers2011local} extends the theory of constructible
functions to allow half-integers.  To accommodate the square roots
introduced by $\rho$, we extend the theory in that way without further
comment.\footnote{We ask whether the difference $\rho_G - \rho_H$,
  for $G$ and an unramified endoscopic group $H$, is always a sum of
  roots.}

\begin{lemma}\label{lemma:satake} 
  There is a $q$-constructible function $s$ on $P^+\times P^+$ that
  specializes to the function $(\lambda,\mu)\mapsto s_{\lambda,\mu}$.
\end{lemma}

\begin{proof} 
  The coefficients $s_{\lambda,\mu}$ are given by an integral of a
  $q$-constructible function on $P^+\times G$:
\begin{equation}\label{eqn:s}
(\lambda,\mu)\mapsto s_{\lambda,\mu}
=\frac{q^{\langle\rho,\mu\rangle}}{\op{vol}(A_0)} 
\int_{A_\mu} \int_N \op{char}(L_G)(\lambda,t n) dn dt.
\end{equation}
Here $N$ is the unipotent radical of a Borel subgroup $B$ containing a
maximally split Cartan subgroup $T$.  The subgroup $A_0 = A\cap K$ is
a maximal compact subgroup of $A$.  Its volume $\op{vol}(A_0)$
specializes to a polynomial in $q$ that can be written as a product of
cyclotomic polynomials.  Adjusting $\op{vol}(A_0)$ by a null function,
we may assume that $\op{vol}(A_0)$ is the specialization of a product
of cyclotomic polynomials.  Cyclotomic polynomials in $q$ are invertible
constructible motivic functions.  Thus, $\op{vol}(A_0)$ can be
inverted.  Integration here is understood to be motivic integration
with respect to invariant volume forms on $N$ and $A$.  The
pushforward under a definable morphism (integration over fibers)
carries $q$-constructible functions to $q$-constructible functions.
Therefore $(\lambda,\mu)\mapsto s_{\lambda,\mu}$ is a
$q$-constructible function.
\end{proof}

\begin{remark} The proof of the previous lemma motivates the following
  question.  Let $E/RF$ be an extension of the residue field sort
  described by coefficients running over a definable subassignment
  $Z$ as in Section \ref{sec:defred}.  Let $T$ be a torus defined over
  the residue field sort that splits over $E$.  The torus is
  classified by $(X^*(T),X_*(T),\theta)$, where $\theta$ is the
  automorphism of the character and cocharacter lattices determined by
  a quasi-Frobenius automorphism.  Let $\op{vol}(T)$ be the motivic
  volume of $T$, viewed as usual as an element of a Grothendieck ring.
  Upon specialization to a finite field $\ring{F}_q$, the cardinality of a torus $T$ is
  given by a determinant \cite[Prop.3.3.7]{carter1985finite}:
\[
D(\theta,q):=\det(\theta q-1; X_*(T)\otimes\ring{Q}).
\]
Is $\op{vol}(T) = D(\theta,\ring{L})$ in the
Grothendieck group, where $\ring{L}$ is the Lefschetz motive?
\end{remark}

\section{Presburger constructibility}

In this section, we check that some functions related to the finite
dimensional representations of complex reductive groups are Presburger
constructible functions on the appropriate integer lattices.  

\begin{remark}\label{rem:matrix}
  For purposes of constructibility, we consider $\ring{Z}^r$ and also
  $Y^*$ as definable subassignments $h[0,0,r]$. When dealing with
  $\ring{Z}^r$, integrals over fibers in the sense of motivic
  integration are discrete sums.  For example, if $(\lambda,\mu)\to
  a_{\lambda,\mu}$ and $(\mu,\nu)\to b_{\mu,\nu}$ are constructible
  functions of integer parameters $(\lambda,\mu,\nu)\in L\times
  M\times N$, then we may interpret the matrix product
  $(\lambda,\nu)\to \sum_{\mu} a_{\lambda,\mu} b_{\mu,\nu}$ as a fiber
  integral as follows.  We pull $a_{\lambda,\mu}$ and $b_{\mu,\nu}$
  both back to $L\times M\times N$, multiply them as constructible
  functions on $L\times M\times N$, then integrate (sum) over the
  fibers of the projection morphism $L\times M\times N\to L\times N$.
\end{remark}

\subsection{weight multiplicities}

We expand the partition function into an infinite series
\[
P(\hat G,V_R,\theta_1,E,q) = 
\sum_\mu (P(\hat G,V_R,\theta_1,E,q),e^\mu) e^{\mu}.
\]

\begin{lemma}\label{lemma:partition}
  The function $\mu\mapsto (P(\hat G,V_R,\theta_1,E,q),e^\mu)$ is
  Presburger constructible.  The function $\mu\mapsto (P(\hat
  G,V_R,\theta_1,E,1),e^\mu)$ is Presburger constructible.
\end{lemma}

\begin{proof} 
  Recall (Lemma \ref{lemma:fact}) that $P(\hat G,V_R,\theta_1,E,q)$ is
  a product of factors of the form
\[
(1- \zeta q^b e^\alpha)^{-1}
\]
where $\zeta$ is a root of unity ($\zeta^k=1$) and $\alpha = N_1\mu$
is a norm.

If $\sum_\mu a_\mu e^\mu$ and $\sum_\mu b_\mu e^\mu$ have
constructible coefficients, then it is easily checked that the product
also has constructible coefficients (see Remark \ref{rem:matrix}).
Thus, the proof reduces immediately to showing that constructibility
of the coefficients of
\[
\frac{1}{1-\zeta q^b e^\alpha} 
= \sum_{j=0}^\infty \zeta^j q^{j b} e^{j\alpha} = 
\sum_{a = 0}^{k-1} \zeta^a 
q^{a b} e^{a\alpha}\sum_{i=0}^\infty q^{i k b} e^{i k\alpha},
\]
with reindexing $j = i k + a$, with $0\le a < k$.  The coefficients in
the inner sum on the right are evidently constructible.
\end{proof}

We continue to work in the usual context of an unramified reductive
group $G$, dual ${}^LG$, and partition function $P(E,q)=P(\hat
G,\n,\theta,E,q)$.

Let $m_{\lambda,\mu}\in \ring{N}$ be the multiplicity of the weight
$\mu\in X^*(T)$ in the irreducible representation with highest weight
$\lambda\in P^+$.

\begin{lemma}  
  The weight multiplicity function $(\lambda,\mu)\mapsto
  (\tau_\lambda,e^\mu)$, the $q$-weight multiplicity function
  $(\lambda,\mu)\mapsto (\tau_{\lambda,q},e^\mu)$ and the inverse
  Satake transform $(\lambda,\mu)\mapsto t_{\lambda,\mu}$ are all
  Presburger constructible functions.
\end{lemma}

\begin{proof} 
  Each function is a finite sum over $w\in W^\theta$ of partition
  functions.  Because constructible functions form a ring, it is
  enough to check that each term in the sum is constructible.  The
  relevant partition functions are $P(E^{-w},1)$, $P(E^{-w},q)$,
  and $P(E^{-w},q^{-1})$, respectively.  These are constructible by
  Lemma \ref{lemma:partition}.
\end{proof}

Recall that $n_{\lambda,\mu}$ is the inverse of the weight
multiplicity matrix.

\begin{theorem}\label{lemma:van-leeuwen} 
  $n_{\mu,\lambda}$ is a Presburger constructible function on
  $\dom\times\dom$.
\end{theorem}

\begin{proof} 
  This is a consequence of van Leeuwen's formula (Lemma
  \ref{lemma:van}).  Referring to that formula, it is enough to show
  constructibility of each term in the sum, with fixed $(w',w)$.  This
  follows from the definability of the set $Y^*_w$ and of the delta
  function $(\mu,\lambda)\mapsto \delta_{w\bullet (w'\mu),\lambda}$.
\end{proof}

\begin{corollary} 
  Consider the geometric Satake transform (\ref{thm:gs}):
\[
\hat f_\lambda = \sum_\mu g_{\lambda,\mu} \tau_\mu.
\]
Then $(\lambda,\mu)\mapsto g_{\lambda,\mu}$ is Presburger
constructible.
\end{corollary}

\begin{proof}  
  The functions $(\lambda,\mu)\mapsto s_{\lambda,\mu}$ and
  $(\lambda,\mu)\mapsto n_{\lambda,\mu}$ are constructible.  The basis
  $g_{\lambda,\mu}$ is the matrix product of these two bases.  The
  result follows from Remark \ref{rem:matrix}: matrix multiplication
  with definable indexing sets preserves constructibility.
\end{proof}

A second proof of the theorem can be obtained from Theorem
\ref{thm:gs}.

\subsection{branching formulas}

While we are on the topic of constructibility, we point out the
constructibility of branching multiplicities.  For example, we have
the following corollary of the classical branching multiplicity formula.

\begin{lemma} 
  Let $H\le G$ be complex reductive groups with Lie algebras
  ${\mathfrak h}\subseteq {\mathfrak g}$.  Fix maximal tori $T_H\le
  T_G$ with Lie algebras $t_h\subseteq t_g$.  Assume that there is an
  element $X_0\in t_h$ such that $\langle\alpha,X_0\rangle>0$ for
  every positive root of ${\mathfrak g}$.  Let $\dom_G$ and $\dom_H$
  be the sets of dominant weights in $G$ and $H$.  Let
  $m(\lambda,\mu)$ be the multiplicity of the irreducible $\mathfrak
  h$-module with highest weight $\mu$ in the irreducible $\mathfrak
  g$-module with highest weight $\lambda$.  Then $m(\lambda,\mu)$ is a
  Presburger constructible function on $\dom_G\times\dom_H$.
\end{lemma}

\begin{proof}  
  Kostant's formula expresses each branching multiplicity as a finite
  sum of partition functions \cite[Theorem ~8.2.1]{goodman}.  Each
  partition function is rational.  Thus, the argument used in the
  proof of Lemma~\ref{lemma:partition} applies.
\end{proof}

\begin{remark}
  Explicit formulas for branching multiplicities are typical of what
  Presburger constructible functions look like.  Typically branching
  formulas look like products of linear factors depending on cases
  that can be described by linear inequalities on dominant weights
  $\lambda$ and $\mu$.  We do not pursue the topic, but we can
  similarly investigate the constructibility of the function giving
  the multiplicities of $\tau_\mu$ in $\op{Sym}^k \tau_\lambda$, and
  related operations on characters.
\end{remark}

Let $(\lambda,\mu)\mapsto m(\lambda,\mu)$ be the function constructed
in Section \ref{sec:branch} that is attached to an embedding
$\xi:{}^LH\to {}^LG$ of endoscopic groups.

\begin{lemma}\label{lemma:branch} 
  The branching multiplicity function $m(\lambda,\mu)$ is Presburger
  constructible on $P^+\times P_H^+$.
\end{lemma}

\begin{proof} 
  By Lemma \ref{lemma:transform}, it is enough to assume that the
  elements $\dotw$ agree (for $\D$ and $\xi$).  (Note that a change of
  $\xi$ changes the weights $m(\lambda,\mu)$ by $\mu(t)$, which is a
  constructible function of $\mu$, whenever $t$ has finite order.  In
  fact, each preimage $\{\mu\mid \mu(t) =\zeta\}$ is Presburger
  definable.)  It is enough to show that each term in Equation
  \ref{eqn:branch} is Presburger constructible.  This reduces to the
  constructibility of the terms $p_{\mu'}$, which follows from the
  rationality of the partition function (Lemma \ref{lemma:partition}).
\end{proof}

\subsection{the constructibility of $B_\xi$}\label{sec:B}

Let ${}^LG$ be the $L$-group of an unramified reductive group
$G$.  Let ${}^LH$ be the dual of an unramified endoscopic group $H$.
We assume that both $H$ and and $G$ are given in the category of
definable subassignments over a cocycle space $Z$.  We can assume that
the cocycle space $Z$ is the same for $H$ and $G$.

Working $p$-adically, Langlands gives a homomorphism
$b = b_\xi$ from the spherical Hecke algebra of $G$ to the spherical
Hecke algebra of $H$.  If $f$ belongs to the spherical Hecke algebra
of $G$, its Satake transform belongs to $\ring{C}[Y^*]^{W^\theta}$.
The biinvariant function $b_\xi(f)\in\H(H//K_H)$ is the inverse
Satake transform of the image of $f$ in
\[
\ring{C}[X^*(\hat T_1)]^{W_H^{\theta_H}}.
\]

\begin{theorem}\label{thm:B}
  Let $G$ be an unramified connected reductive group with unramified
  endoscopic group $H$, both considered as definable subassignments
  over a cocycle space $Z$.  Fix an $L$-embedding $\xi:{}^LH\to {}^LG$
  that factors through a finite cyclic group $\ang{\theta}$; that is,
  $\xi:\hat H\rtimes \ang{\theta_H} \to \hat G \rtimes \ang{\theta}$.
  Then there is a $q$-constructible function $B_\xi$ on $P^+_G\times H$
  and a natural number $N$ with the following specializations:
\[
B_\xi(\lambda,h)_F = b_\xi(f_{F,\lambda})(h),\quad \text{for } h\in H(F),
\]
for all $p$-adic fields in $F\in C_N$.  
\end{theorem}

Recall that for each $F$, we let $f_{F,\lambda}$ denote the
characteristic function of the double coset $K\varpi_F^\lambda K$ in
the unramified reductive group $G$ over $F$.  The theorem implies that
the homomorphism $b_\xi$ has uniform behavior as the $p$-adic field
varies, and as $\lambda$ varies.

\begin{proof}
  We have done most of the work already for this theorem.  Let $L_G$
  and $L_H$ be the definable sets given in Lemma \ref{lemma:cartan}
  for $G$ and $H$.  
   We have
  \begin{align*}
    J_\xi \hat f_\lambda  &= 
\sum  g_{\lambda,\lambda'} J_\xi\tau_{\lambda'}
\\ &=
\sum g_{\lambda,\lambda'} m(\lambda',\mu) \sigma_\mu 
\\ &=
\sum g_{\lambda,\lambda'} m(\lambda',\mu) t^H_{\mu,\mu'} \hat f_{\mu'}^H.           
    \end{align*}
As usual $g_{\lambda,\lambda'}$ are the coefficients of the geometric
Satake
transform, $m(\lambda',\mu)$ are branching coefficients, and
$t^H_{\mu,\mu'}$ are coefficients of the inverse Satake transform on
${}^LH$.  As proved above, these are all constructible functions of
their
lattice parameters.
Then
\[
B_\xi(\lambda,h) = b_\xi(f_\lambda)(h) =
\sum_{\lambda',\mu,\mu'} g_{\lambda,\lambda'} m(\lambda',\mu)
t^H_{\mu,\mu'}
\op{char}(L_H)(\mu',h).
\]
The sums run over bounded definable sets and are represented by
discrete motivic sums over the integer sort (Remark \ref{rem:matrix}).  
The right-hand side of
this
equation is therefore a constructible function of $(\lambda,h)\in
P^+\times H$.
\end{proof}

\begin{remark}  
  We have stated $q$-constructibility results in terms of the limiting
  behavior on $p$-adic fields $\C_N$ for $N$ large.  However, in fact,
  the formulas we obtain for $B_\xi$ hold for all $p$-adic fields.
\end{remark}

\subsection{transfer principle}\label{sec:transfer}

We review the transfer principle from
\cite{cluckers2010constructible}.

\begin{theorem}\label{thm:transfer-principle}
  Let $f\in C(S)$ be a constructible function on a definable set $S$.
  There exists $N$ such that for all pairs of fields $F_1,F_2\in \C_N$
  with the same residue field, $f_{F_1}$ is identically zero on
  $S(F_1)$ iff $f_{F_2}$ is identically zero on $S(F_2)$.
\end{theorem}

This theorem allows us to transfer identities of motivic integrals of
constructible functions from a field $F_1$ of one characteristic to a
field $F_2$ of another characteristic provided that they have the same
residue field.  The constant $N$ depends on $f$ and is not explicit.

An easy corollary is a transfer principle for statements $\psi$ in the
Denef-Pas language.

\begin{corollary} Let $\psi$ be a statement in the Denef-Pas language.
  Then there exists $N$ such that for all pairs of fields $F_1,F_2\in
  \C_N$ with the same residue field, $\psi_{F_1}$ holds iff
  $\psi_{F_2}$ holds.
\end{corollary}

\begin{proof} Let $f\in C(\op{pt})$ be the function on a point that is
  the characteristic function of $\psi$.  Apply the transfer principle
  to $f$.
\end{proof}

\subsection{fundamental lemma}

We conclude this article with a proof of the fundamental lemma for the
spherical Hecke algebra for unramified groups in large positive
characteristic in the following form.

\begin{theorem} \label{thm:fl} For each absolute root system $R$,
  there is a constant $N=N_R\in\ring{N}$ such that the
  Langlands-Shelstad fundamental lemma for the spherical Hecke algebra
  $\H(G//K)$ holds for all unramified connected reductive groups $G$
  with absolute root system $R$ and all of its unramified endoscopic
  groups $H$ over $F$ for all $p$-adic fields $F\in \C_N$.
\end{theorem}


\begin{proof}
  We assume that the reader is familiar with the proof that the
  fundamental lemma for the unit element of the Hecke algebra can be
  transferred from one field to another by the transfer principle of
  motivic integration \cite{cluckers2011transfer}.  The method is the
  same here.  Once we establish that the fundamental lemma can be
  expressed as an identity of constructible functions, then the
  machinery of motivic integration and the transfer principle takes
  over and gives the theorem.  Earlier work has already shown how
  orbital integrals can be expressed as motivic integrals of
  constructible functions.  Our proof therefore reduces to checking
  the constructibility of the Langlands-Shelstad transfer factor and
  to checking the constructibility of the function $B_\xi$.

  The fundamental lemma takes the form
\begin{equation}\label{eqn:fl}
\sum_{\gamma_G}\Delta_0(\gamma_H,\gamma_G,\cdots)
\op{O}(\gamma_G,f_\lambda) - \op{SO}(\gamma_H,b_\xi(f_\lambda)) = 0.
\end{equation}
Stable orbits of regular semisimple elements are definable as fibers
of the Chevalley morphism $G\to T/W$.  The invariant motivic measure
on stable orbits is the volume form attached to a Leray residue of an
invariant differential form on the group with respect to the canonical
form on $T/W$.  

The constructibility of the transfer factor is treated in Appendix
\ref{sec:xfer}.  The ellipsis $(\cdots)$ in the transfer factor
indicates extra free parameters such as a parameter running over
$a$-data, a parameter running over admissible pinnings for the
canonical normalization, and uniformizing parameters used in our
explicit treatment of the $\chi$-data.  The $p$-adic transfer factor
is independent of these choices, but in dealing with constructible
motivic functions, it is best to make the dependence on the parameters
explicit (or at least honor them with an ellipsis).

The homomorphism $b_\xi(f_\lambda)$ can be replaced with
the constructible function $B_\xi$.

We may consider the left-hand side of Equation \ref{eqn:fl} as a
$q$-constructible function of $(\lambda,\gamma_H,\cdots)\in P^+\times
H\times\cdots$, all over a definable cocycle space $Z$ used to
parameterize an unramified splitting field of $G$ and $H$.

The fundamental lemma holds for the unit element in positive
characteristic by the work of Ng\^o \cite{ngo2010lemme}.  This can be
lifted to characteristic zero \cite{cluckers2011transfer},
\cite{waldspurger2006endoscopie}.  It extends to the full Hecke
algebra in characteristic zero \cite{hales1995fundamental}.  Hence the
identity (\ref{eqn:fl}) holds in characteristic zero.  By the transfer
principle, there exists $N$ such that the fundamental lemma also holds
for all fields $F\in\C_N$.

Furthermore, the arguments of \cite{hales1995fundamental} reduce the
fundamental lemma for the Hecke algebra (in characteristic zero or
large positive characteristic) to $G = G_{adj}$.  For an adjoint
group, there are only finitely many choices of unramified $G$ and $H$
up to equivalence associated with a given root system and only
finitely many choices of $L$-morphisms $\xi$ that satisfy our
conditions.  Thus, we can arrange for $N$ to depend only on the root
system.

It is important for the left-hand side of the equation to be viewed as
a single identity with $P^+$ forming a factor of the definable
subassignment, rather than viewed as an infinite collection of
identities indexed by $\lambda\in P^+$.  This allows us to invoke the
transfer principle a single time, rather than once for each
$\lambda\in P^+$.
\end{proof}

\section{Appendix on transfer factors}\label{sec:xfer}

In this section, we assume familiarity with the Langlands-Shelstad
transfer factor \cite{langlands1987definition}.

In \cite{gordon}, we showed that the Lie algebra transfer factor is a
constructible motivic function.  In that article, by restricting
attention to a small neighborhood of the identity element of the
group, we were able to avoid the analysis of multiplicative characters
that appear in the group-level transfer factor.  In this appendix we
analyze the multiplicative characters and prove that group-level
transfer factor is constructible for unramified endoscopic data.

We use the canonical normalization of transfer factors given in
\cite[\S7]{hales1993simple}.  The canonical normalization requires a
choice of an admissible pinning.  The admissible pinning involves a
choice of simple root vectors $X_\alpha$ (with respect to a fixed
Borel subgroup and Cartan).  The choices $X_\alpha$ range over a
definable subassignment, and we obtain the canonical normalization by
introducing a free parameter into the transfer factor ranging over the
definable subassignment.

\subsubsection{$a$-data}

To define the transfer factor for $p$-adic fields, a choice of
$a$-data is made, but the transfer-factor is in fact independent of
the choice of $a$-data.

This section introduces a definable subassignment of $a$-data and
introduces an explicit free variable $a$ into the transfer factor that
ranges over the definable subassignment of $a$-data.  

We begin with a review of $a$-data for a $p$-adic field, then show how
to make the construction as a definable subassignment.  Let $\Gamma$
be the Galois group of a Galois extension $L/F$.  We assume that
$\Gamma$ acts on a finite set $R$ of roots.  The $a$-data are a
collection of constants $a_\alpha\in L^\times$ indexed by $\lambda\in
R$ such that
\begin{equation}\label{eqn:a}
a_{-\lambda} = -a_\lambda,\quad a_{\sigma\lambda} 
= \sigma(a_\lambda),\quad \text{ for } \sigma\in \Gamma.
\end{equation}
Let $\epsilon:R\to R$, given by
$\epsilon(\lambda)=-\lambda\ne\lambda$.  Let $O$ be the orbit of some
$\lambda\in R$ under $\langle\Gamma,\epsilon\rangle$.  The choice of
$a$-data can clearly be made orbit by orbit.  If there is no
$\sigma\in \Gamma$ such that $\sigma\lambda=-\lambda$, we have a
specific choice of $a$-data (selecting a given $\lambda\in O$) given
by
\[
a_{\sigma\lambda}=1,
\quad a_{-\sigma\lambda}=-1,\quad \sigma\in\Gamma.
\]
If some $\sigma_0\in\Gamma$ gives $\sigma_0\lambda=-\lambda$, then we
proceed as follows. Let $F_{+\lambda}$ be the fixed field of
$\Gamma_{+\lambda} = \{\sigma\in\Gamma\mid \sigma\lambda=\lambda\}$
and we let $F_{\pm\lambda}$ be the fixed field of $\Gamma_{\pm\lambda}
= \{\sigma\in\Gamma\mid \sigma\lambda=\pm\lambda\}$.  The extension
$F_{+\lambda}/F_{\pm\lambda}$ is quadratic.  We may choose $a$-data by
choosing $a_\lambda\in F_{+\lambda}$ such that $\sigma_0(a_\lambda) =
-a_\lambda$ then extending uniquely to the entire orbit by the
relation (\ref{eqn:a}).  Specifically, when the quadratic extension is
unramified, the choice of $a_\lambda$ can be taken to run over units
of $F_{+\lambda}$ such that its square is a nonsquare in
$F_{\pm\lambda}$.  When the quadratic extension is ramified, we take
$a_\lambda$ to run over uniformizers in $F_{\lambda}$ such that its
square lies in $F_{\pm\lambda}$.  We see by these explicit
descriptions that the tuple $(a_\lambda)$, indexed by $\lambda$, is a
parameter in a definable subassignment.

\subsubsection{$\Delta_{II}$}
Two terms in the transfer factor rely on multiplicative characters
constructed from $\chi$-data: the terms $\Delta_{II}$ and the term
$\Delta_2$.

\begin{lemma}  
  For unramified endoscopic data,
  there is a $q$-constructible function representing $\Delta_{II}$
  (after introducing some free parameters ranging over definable
  subassignments).
\end{lemma}

\begin{proof}  
  We begin with an explicit construction of some characters for a
  $p$-adic field.  Then we analyze the construction to see that it can
  be done constructibly.

  Let $F_+/F_\pm$ be a quadratic extension of $p$-adic fields with
  large residue characteristic.  Let
  $\varpi_+$ be a uniformizer in $F_+$.  We define a multiplicative
  character $\chi_+ = \chi_{F_+/F_{\pm}}:F_+\to \ring{C}^\times$ as
  follows.  If $F_+/F_\pm$ is unramified, let $\chi_+$ be the
  unramified character of order two.

  If $-1$ is a square in $F_+$, we define $\chi_+$ by
  $\chi_+(\varpi_+) = i\in\ring{C}$, and $\chi_+$ restricted to units
  is the unique character of order two.

  If $-1$ is a nonsquare in $F_+$, we define $\chi_+$ by
  $\chi_+(\varpi_+)=1$, and $\chi_+$ restricted to units is the unique
  character of order two.

  In every case, $\chi_+^4 = 1$.

  Now we analyze constructibility.  The condition that $-1$ is a
  square or nonsquare is a definable condition.  Assume that $F_+$ and
  $F_\pm$ are both extensions of $VF$, presented as usual by a
  definable space of the characteristic polynomial of a generator of
  the fields.  Introduce a free parameter $\varpi_+$ that runs over
  the constructible subassignment of uniformizers in $F_+$.  We claim
  that $\chi_+$ is a linear combination of characteristic functions
\[
\chi_+ = \sum_{\zeta\in\mu_4(\ring{C})} \zeta\, \op{char}(D_\zeta).
\]
where each $D_\zeta$ is constructible over the space of parameters.
This is essentially obvious: $F_+/F_\pm$ being unramified is a
definable condition on the coefficients of the characteristic
polynomial; the unique character of order two is given in terms of the
characteristic function on squares and nonsquares, etc.

Now we turn to the transfer factor $\Delta_{II}$.  It has the form
\[
\prod_\alpha \chi_\alpha
\left(\frac{\alpha(\gamma)-1}{a_\alpha}\right),
\]
where $\gamma$ is strongly regular semisimple.
It is a constructible function if each factor is a constructible
function. Each morphism $\gamma\to(\alpha(\gamma)-1)/a_\alpha$ is
definable, so we only need to check that each character $\chi_\lambda$
in some choice of $\chi$-data is constructible.  We use the characters
given above to do so.

There is no harm in partitioning the domain of $\Delta_{II}$ into
finitely many parts according to definable characteristics of the
element $\gamma$.  We consider a an extension $L/VF$ that splits the
centralizer of $\gamma$.  We may assume fixed abstract Galois data
$1\to\Gamma^t\to\Gamma\to \Gamma^u\to 1$ with enumeration $\sigma_i$
of the elements of $\Gamma$ for $L$ and we may assume a fixed action
of that data on the root system coming from the centralizer of
$\gamma$ (relative to a split torus).  This gives the indexing set $R$
of roots and action of $\Gamma$ as fixed choices used to partition the
domain of $\Delta_{II}$.

Let $\epsilon$ be an automorphism of $R$ that acts as $\lambda\mapsto
-\lambda$, and let $O(\lambda)$ be the orbit of $\lambda$ under
$\langle\Gamma,\epsilon\rangle$.  If there does not exist
$\sigma\in\Gamma$ such that $\sigma\lambda=-\lambda$, then we may take
the $\chi$-data for $\mu\in O(\lambda)$ to be the trivial character
(which is constructible).

Now assume that there exists $\sigma\in\Gamma$ such that
$\sigma\lambda = -\lambda$.  Then $F_{+\lambda}/F_{\pm \lambda}$ is a
nontrivial quadratic extension.  We set $\chi_\lambda = \chi_+$ for
this quadratic extension.  In more detail, we include free parameters
$\dot\sigma$ realizing each abstract automorphism $\sigma$ as a linear
transformation of $L/VF$.  The extension $F_{+\lambda}/VF$ and the
space of uniformizers $\varpi_+$ in the extension are then described
by definable conditions inside $L/VF$ (as in
\cite{cluckers2011transfer}).

By transport of structure, we obtain constructible $\chi$-data on the
entire orbit of $\lambda$, using the defining properties of
$\chi$-data: $\chi_{\sigma\lambda} = \chi_{\lambda}\cdot
\dot\sigma^{-1}$; $F_{+\sigma\lambda}=\dot\sigma F_{+\lambda}$;
$\varpi_{+\sigma\lambda}=\dot\sigma\varpi_{+\lambda}$, and so forth.
Running over all orbits this way, constructible $\chi$-data are
obtained.
\end{proof}

\subsubsection{$\Delta_2$}

We have now treated all terms except $\Delta_2$.  We recall that the
term $\Delta_2$ restricts to a multiplicative character on each Cartan
subgroup of $G$.  It is constructed from $\chi$-data by means of class
field theory reciprocity for tori.  The following theorem completes
our analysis of the transfer factors on groups.

\begin{theorem}\label{thm:delta2}  For unramified endoscopic data,
  there is a $q$-constructible function representing the transfer
  factor $\Delta$, possibly after introducing some free parameters.
  These parameters have no effect upon specialization to a $p$-adic
  field.
\end{theorem}

\begin{proof}  
  The idea of the proof is that
  multiplicative characters can be chosen to be tamely ramified; that
  is, they have trivial restriction to topologically unipotent
  elements.  We have descent formulas for unramified groups that
  reduce the transfer factor to the case of topologically unipotent
  elements \cite{langlands2007descent} \cite{hales1993simple}
  \cite{langlands2007descent}.
  We freely use various lemmas on definability from Section
  \ref{sec:definability}.

  We enumerate the standard Levi components of $G$.  Each is a
  definable set.  If $\gamma_G$ is conjugate to an element $\gamma_M$
  in some proper Levi subgroup, then by descent formulas for transfer
  factors we have $\Delta(\gamma_G,\gamma_H) =
  \Delta_M(\gamma_M,\gamma_{M,H})$.  The element $\gamma_{M,H}$ is a
  conjugate of $\gamma_H$ in a Levi of $H$ constructed by descent.  By
  an induction on the dimension of the group, we may assume that
  $\Delta_M$ is constructible.  Every regular semisimple element that
  is not elliptic is conjugate to an element of a proper Levi
  subgroup.  We may now assume that $\gamma_G$ belongs to an elliptic
  Cartan subgroup $T$.

  Since $G$ is unramified, the connected center $Z^0 = Z(G)^0$ is also
  unramified and can be naturally identified with a torus in the
  center of $H$.  By Langlands and Shelstad, there is a character
  $\chi_Z$ on $Z^0$ such that
\[
\Delta(z\gamma_G,z\gamma_H) = \chi_Z(z)\Delta(\gamma_G,\gamma_H).
\]
The character is unramified \cite{hales1993simple}.  The character
$\chi_Z$ depends on $(\gamma_G,\gamma_H)$ only through the endoscopic
data $(G,H)$ and $\xi$.  

We claim that $Z^0$ is definable and that $\chi_Z$ is a constructible
function on $Z^0$.  The connected center $Z^0$ is definable as the
kernel of $G\to G_{ss}$, where $G_{ss}$ is a semisimple quotient of
$G$.  The root datum for $G_{ss}$ can be described as 
fixed choice in terms of the root datum for $G$.
There exists $\xi_0:{}^LH\to {}^LG$ such that $\chi_Z$ is
trivial on $Z^0$ \cite[Lemma 3.6]{hales1995fundamental}.  As we vary
the embedding $\xi:{}^LH\to{}^LG$, the transfer factor changes by an
unramified character that comes from an element
\[
\hat z\in Z(\hat H) \subseteq \hat T \to \hat T_1.
\]
It is enough to show that this unramified character is constructible.
We treat this element $\hat z$ and its image $\tau\in \hat T_1$ as a
fixed choice.  Following Equation \ref{eqn:identify}, we have
\[
\op{Hom}(T_H/T_H(O),\ring{C}^\times) 
= \hat T_1 = \op{Hom}(X^*(A_H),\ring{C}^\times).
\]
The element $\tau$ has finite order.  The character $\chi_Z$ is the
restriction of this character to $Z^0 \subseteq T_H$.  In the
constructible context, we start with $\hat z$ and $\tau$, writing
$\chi_Z$ on $T_H$ as
\[
\sum_{\mu\in X^*(A_H)} \mu(\tau)\op{char}(A_{H,\mu}),
\]
where $A_{H,\mu}$ is the definable set of Equation \ref{eqn:alambda}
for $H$.  This clearly restricts to a constructible function on $Z^0$.

By adjusting $(\gamma_G,\gamma_H)$ by a central element in $Z^0$, we
may reduce the proof of constructibility to the special case where
$\gamma_G$ and $\gamma_H$ lie in the maximal bounded subgroup of their
Cartan subgroups $T$ and $T_H$.  We take a definable topological
Jordan decomposition $\gamma_G = \gamma_s \gamma_u$, as described in
Section \ref{sec:definability}.  Replacing $\gamma_G$ and $\gamma_s$
by stable conjugates $\gamma'_G=\gamma_G^h$ and $\gamma_s^h$ (same
$h$), we may assume that $\gamma_s\in G_u$; that is, its centralizer
is unramified.  We may do the same on the endoscopic side.  We have
\[
\Delta_G(\gamma_G,\gamma_H) = c \Delta_G(\gamma'_G,\gamma'_H)
\]
where $c$ is the ratio of terms coming from $\Delta_{III}$.  The
$\Delta_{III}$ terms are constructible, so the proof of
constructibility reduces to the case where we may now drop primes and
assume that $\gamma_s$ has an unramified centralizer.  We construct
descent data $(G_s,H_s)$ for the centralizer of $\gamma_s$ in $G$ and
the corresponding centralizer in $H$.  By \cite{hales1993simple}, the
normalized transfer factors satisfy
\[
\Delta(\gamma_G,\gamma_H) = \Delta_s(\gamma_G,\gamma_H),
\]
where the right-hand side is computed with respect to the endoscopic
data $(G_s,H_s)$.  (In that reference, it is assumed that $\gamma_G\in
G(O_F)$ and $\gamma_H\in H(O_F)$, but that assumption is only needed
to prove the fact that the centralizer of $\gamma_s$ is unramified.
Since we have a separate argument for that fact, the descent formula
holds in our context as well. That reference also uses the topological
Jordan decomposition rather than our definable version, which is not
unique.  It can be checked that the formulas for the transfer factor are
insensitive to the choice of decomposition $\gamma=\gamma_s\gamma_u$.)
By an induction on the dimension of the group, the right-hand side is
constructible, and the proof is complete except in the case when
$\gamma_s$ is central.

We now assume that $\gamma_s$ is central and strongly compact.  Then
$\gamma_s\in K$ because $K$ is a maximal compact.  It is known that
$\chi_Z$ is trivial on $K$ \cite[Lemma
3.2]{hales1995fundamental}. Thus again adjusting by an element in the
center, we may assume that $\gamma_s=1$.  That is, we are reduced to
proving the constructibility of transfer factors on the set of
topologically unipotent elements.  We pick our $\chi$-data to be
tamely ramified.  This implies that the characters $\Delta_2$ are
trivial on topologically unipotent elements.  This reduces
the constructibility of $\Delta$ to the analysis of factors $\Delta_I$, $\Delta_{II}$,
$\Delta_1$, and $\Delta_{IV}$.  This has already been done.  This
completes the proof.
\end{proof}

      \bibliography{hecke} 
      
      \bibliographystyle{plainnat}

\end{document}